%% file: koszul.tex
\documentclass{article}
\usepackage{amssymb}
\usepackage{amsmath, amsthm, amscd, amsfonts}
\usepackage{mathtools}
\usepackage{diagbox}
\usepackage{tensor}
\usepackage{verbatim}

\newtheorem{theorem}{Theorem}
\newtheorem{definition}{Definition}

\newtheorem{corollary}[theorem]{Corollary}
\newtheorem{proposition}{Proposition}

\input{strands_def}
\newcommand{\iz}{\mathbb{I}_\mathcal{Z}}
\newcommand{\slz}{\mathcal{Z}}
\newcommand{\sla}{\mathcal{A}}
\newcommand{\widehatit}[1]{\widehat{\mathit{#1}}}
\begin{document}

\title{Explicit Koszul-dualizing bimodules in bordered Heegaard Floer homology}
\author{Bohua Zhan}
\date{May 21, 2015}
\maketitle
\begin{abstract}
  We give a combinatorial proof of the quasi-invertibility of
  $\widehatit{CFDD}(\iz)$ in bordered Heegaard Floer homology, which implies a
  Koszul self-duality on the dg-algebra $\sla(\slz)$, for each pointed matched
  circle $\slz$. This is done by giving an explicit description of a rank 1
  model for $\widehatit{CFAA}(\iz)$, the quasi-inverse of
  $\widehatit{CFDD}(\iz)$. This description is obtained by applying homological
  perturbation theory to a larger, previously known model of
  $\widehatit{CFAA}(\iz)$.
\end{abstract}

\input{koszul12}
\input{koszul34}
\input{koszul5}
\bibliographystyle{plain}
\bibliography{paper}

\input{koszul_appendix}

\end{document}

%% file: strands_def.tex
\usepackage{tikz}
\usepackage{subfig}
\usepackage{ifthen}
\usetikzlibrary{arrows}

\newcommand{\showpmcdots}[3]{
\foreach \x in {1,...,#3} {
	\filldraw (#1,\x) circle (0.1);
	\filldraw (#2,\x) circle (0.1);
}}

\newenvironment{smallsd}[1][7pt]
{\begin{tikzpicture} [x=#1,y=#1,baseline=(current bounding box.center)]
\showpmcdots{0}{3}{4}}
{\end{tikzpicture}}
\newenvironment{sdn}[2][7pt]
{\begin{tikzpicture} [x=#1,y=#1,baseline=(current bounding box.center)]
\showpmcdots{0}{3}{#2}}
{\end{tikzpicture}}

{\begin{tikzpicture} [x=#1,y=#1,baseline=(current bounding box.center)]
\showpmcdots{0}{3}{6}
\filldraw[color=white] (0,3) circle (0.2);
\filldraw[color=white] (3,3) circle (0.2);
\filldraw[color=white] (0,4) circle (0.2);
\filldraw[color=white] (3,4) circle (0.2);}
{\end{tikzpicture}}

{\begin{tikzpicture} [x=#1,y=#1,baseline=(current bounding box.center)]
\showpmcdots{0}{3}{8}
\filldraw[color=white] (0,4) circle (0.2);
\filldraw[color=white] (3,4) circle (0.2);
\filldraw[color=white] (0,5) circle (0.2);
\filldraw[color=white] (3,5) circle (0.2);
\filldraw[color=white] (0,6) circle (0.2);
\filldraw[color=white] (3,6) circle (0.2);}
{\end{tikzpicture}}

\newcommand{\strandup}[2]{\draw [->] (0,#1) to [out=0,in=180] (3,#2);}
\newcommand{\stranddown}[2]{\draw [->] (3,#2) to [out=180,in=0] (0,#1);}
\newcommand{\doublehor}[2]{\draw [dashed] (0,#1) to (3,#1); \draw [dashed] (0,#2) to (3,#2);}
\newcommand{\singlehor}[1]{\draw [dashed] (0,#1) to (3,#1);}

\newenvironment{smalldsd}[1][7pt]
{\begin{tikzpicture} [x=#1,y=#1,baseline=(current bounding box.center)]
\showpmcdots{3}{3}{4}
\draw (3,0) to (3,5);}
{\end{tikzpicture}}
\newenvironment{dsdn}[2][7pt]
{\begin{tikzpicture} [x=#1,y=#1,baseline=(current bounding box.center)]
\showpmcdots{0}{6}{#2}
\draw (3,0) to (3,#2) to +(0,1);}
{\end{tikzpicture}}

\newenvironment{subsd}[2]
{\begin{scope} [shift={(#1,#2)}]
\showpmcdots{3}{3}{8}
\draw (3,0) to (3,9);\draw[white](0,0) to (0,9);}
{\end{scope}}

\newenvironment{smallsubsd}[2]
{\begin{scope} [shift={(#1,#2)}]
\showpmcdots{3}{3}{4}
\draw (3,0) to (3,5);\draw[white](0,0) to (0,5);}
{\end{scope}}

\newenvironment{smallsubssd}[2]
{\begin{scope} [shift={(#1,#2)}]
\showpmcdots{0}{3}{4}}
{\end{scope}}

\newcommand{\lstrand}[2]{\draw (0,#1) to [out=0,in=180] (3,#2);}
\newcommand{\rstrand}[2]{\draw (3,#1) to [out=0,in=180] (6,#2);}

\newcommand{\ldoublehor}[2]{\draw [dashed] (0,#1) to (3,#1);\draw [dashed] (0,#2) to (3,#2);}
\newcommand{\rdoublehor}[2]{\draw [dashed] (3,#1) to (6,#1);\draw [dashed] (3,#2) to (6,#2);}
\newcommand{\lsinglehor}[1]{\draw [dashed] (0,#1) to (3,#1);}
\newcommand{\rsinglehor}[1]{\draw [dashed] (3,#1) to (6,#1);}

\newcommand{\keypair}[3]{
\ifthenelse{\equal{#3}{left}}{\draw [dashed] (-0.5,#1) .. controls (-1.5,#1) and (-1.5,#2) .. (-0.5,#2);}{}
\ifthenelse{\equal{#3}{right}}{\draw [dashed] (6.5,#1) .. controls (7.5,#1) and (7.5,#2) .. (6.5,#2);}{}
}

\newcommand{\setarrows}[4]{
	\def\upperarrow{#1}
	\def\leftarrow{#2}
	\def\rightarrow{#3}
	\def\lowerarrow{#4}
}
\newcommand{\resetarrows}{\setarrows{$H$}{$d$}{$d$}{$H$}}

\newcommand{\sethspecial}{\setarrows{$H_{sp}$}{$d$}{$d$}{$H_{sp}$}}

\resetarrows

\newcommand{\canceldiag}[6]{
\begin{tikzpicture} [x=8pt,y=8pt]
\begin{subsd}{0}{14}\ifthenelse{\not\equal{#1}{#2}}{\keypair{#1}{#2}{left}}{}#3\end{subsd}
\begin{subsd}{12}{14}\ifthenelse{\not\equal{#1}{#2}}{\keypair{#1}{#2}{right}}{}#4\end{subsd}
\begin{subsd}{0}{0}\ifthenelse{\not\equal{#1}{#2}}{\keypair{#1}{#2}{left}}{}#5\end{subsd}
\begin{subsd}{12}{0}\ifthenelse{\not\equal{#1}{#2}}{\keypair{#1}{#2}{right}}{}#6\end{subsd}
\draw [->] (7,18.5) to node [above] {\upperarrow} (11,18.5);
\draw [->] (7,4.5) to node [above] {\lowerarrow} (11,4.5);
\draw [->] (3,13) to node [right] {\leftarrow} (3,10);
\draw [->] (15,13) to node [right] {\rightarrow} (15,10);
\end{tikzpicture}
}

\newcommand{\mapdiag}[5]{
\begin{tikzpicture} [x=10pt,y=10pt,baseline=(current bounding box.center)]
\begin{subsd}{0}{0}\ifthenelse{\not\equal{#1}{#2}}{\keypair{#1}{#2}{left}}{}#4\end{subsd}
\begin{subsd}{12}{0}\ifthenelse{\not\equal{#1}{#2}}{\keypair{#1}{#2}{right}}{}#5\end{subsd}
\draw [->] (7,4.5) to node [above] {#3} (11,4.5);
\end{tikzpicture}
}

\newcommand{\doublemapdiag}[7]{
\begin{tikzpicture} [x=8pt,y=8pt]
\begin{subsd}{0}{0}\ifthenelse{\not\equal{#1}{#2}}{\keypair{#1}{#2}{left}}{}#5\end{subsd}
\begin{subsd}{12}{0}#6\end{subsd}
\begin{subsd}{24}{0}\ifthenelse{\not\equal{#1}{#2}}{\keypair{#1}{#2}{right}}{}#7\end{subsd}
\draw [->] (7,4.5) to node [above] {#3} (11,4.5);
\draw [->] (19,4.5) to node [above] {#4} (23,4.5);
\end{tikzpicture}
}

%% file: koszul12.tex
\section{Introduction}

Bordered Heegaard Floer homology, introduced by Robert Lipshitz, Peter
Ozsv\'{a}th, and Dylan Thurston in \cite{LOT08}, \cite{LOT10a}, gives a way of
extending the hat version of Heegaard Floer homology to 3-manifolds with
boundary. Besides its theoretical interest, it has shown to be an effective
computational tool, for example in giving an efficient, algorithmic way to
compute the Heegaard Floer homology of any 3-manifold \cite{LOT10c}.

The theory associates invariants to 3-manifolds with parametrized boundaries. A
parametrization of the boundary is a diffeomorphism from the boundary to some
standard genus $g$ surface. A standard genus $g$ surface is, in turn, described
by a pointed matched circle, which can be considered as a handle decomposition
of the surface. For a pointed matched circle $\slz$, we denote by $F(\slz)$ the
standard surface parametrized by $\slz$. Let $-\slz$ denote $\slz$ with reversed
orientation, then $F(-\slz)$ is the orientation reversal of $F(\slz)$.

To each pointed matched circle $\slz$, the theory associates a combinatorially
defined dg-algebra $\sla(\slz)$. To every 3-manifold $Y$ with boundary
parametrized by $\slz$, it associates two invariants: the type $A$ invariant
$\widehatit{CFA}(Y)_{\sla(\slz)}$, which is a right $A_\infty$-module over
$\sla(\slz)$, and the type $D$ invariant $^{\sla(-\slz)}\widehatit{CFD}(Y)$,
which is a left type $D$ structure over $\sla(-\slz)$ (we will briefly review
these algebraic concepts in the next section). The modules are well-defined up
to homotopy equivalence (denoted by $\simeq$).

The type $D$ and type $A$ invariants are related to each other by taking the box
tensor product, $\cdot\boxtimes\cdot$, with one of two special bimodules:
$\widehatit{CFDD}(\iz)$ and $\widehatit{CFAA}(\iz)$. They are called the identity
type $\mathit{DD}$ and type $\mathit{AA}$ bimodules, respectively. More
explicitly, the relations are:
\begin{eqnarray}
  \widehatit{CFA}(Y)_{\sla(\slz)} \simeq
  \widehatit{CFAA}(\iz)_{\sla(-\slz),\sla(\slz)} \boxtimes 
  \tensor*[^{\sla(-\slz)}]{\widehatit{CFD}(Y)}{} \\
  ^{\sla(-\slz)}\widehatit{CFD}(Y) \simeq
  \widehatit{CFA}(Y)_{\sla(\slz)} \boxtimes 
  \tensor*[^{\sla(\slz),\sla(-\slz)}]{\widehatit{CFDD}(\iz)}{}.  
\end{eqnarray}

The bimodule $\widehatit{CFDD}(\iz)$ is quasi-invertible, with
$\widehatit{CFAA}(\iz)$ being its quasi-inverse (that is,
$\widehatit{CFDD}(\iz)\boxtimes \widehatit{CFAA}(\iz)\simeq \mathbb{I}$, where
$\mathbb{I}$ denotes the identity type $\mathit{DA}$ bimodule). This implies
that taking box tensor product with $\widehatit{CFDD}(\iz)$ induces an
equivalence of categories between the category of right $A_\infty$-modules over
$\sla(\slz)$ and the category of type $D$ structures over $\sla(-\slz)$. In
particular, the invariants $\widehatit{CFA}(Y)_{\sla(\slz)}$ and
$^{\sla(-\slz)}\widehatit{CFD}(Y)$ actually contain the same information about
$Y$.

Both $\widehatit{CFAA}(\iz)$ and $\widehatit{CFDD}(\iz)$ are defined by
holomorphic curve counts. While $\widehatit{CFDD}(\iz)$ can be described
combinatorially (in \cite{LOT10b} and \cite{LOT10c}), the quasi-invertibility of
$\widehatit{CFDD}(\iz)$ is verified in \cite{LOT10a} using holomorphic curve
methods.

The aim of this paper is to describe combinatorially the type $\mathit{AA}$
invariant $\widehatit{CFAA}(\iz)$. More precisely, we construct an explicit rank
1 $A_\infty$-bimodule $\mathcal{N}$ with the following property:

\begin{theorem} \label{thm:nboxcfddprop}
  $\mathcal{N}\boxtimes\widehatit{CFDD}(\iz)$ has rank 1. Furthermore,
  $\mathcal{N}$ is quasi-invertible, in the sense that there exists a type
  $\mathit{DA}$ bimodule $\mathcal{N'}$, such that $\mathcal{N}\boxtimes
  \mathcal{N'}$ is homotopy equivalent to the identity bimodule ($\mathcal{N'}$
  is called the quasi-inverse of $\mathcal{N}$).
\end{theorem}
This quickly leads to a combinatorial proof that $\widehatit{CFDD}(\iz)$ itself
is quasi-invertible. By construction, the bimodule $\mathcal{N}$ is in the
homotopy class of $\widehatit{CFAA}(\iz)$, so we know from analysis that it is
the actual quasi-inverse of $\widehatit{CFDD}(\iz)$. This stronger statement
will be proved combinatorially in \cite{BZ2}.

Considering $\widehatit{CFDD}(\iz)$ as a left-right type $\mathit{DD}$ bimodule
$^{\sla(\slz)}\widehatit{CFDD}(\iz)^{\sla(\slz)}$, its quasi-invertibility also
implies a kind of Koszul self-duality of $\sla(\slz)$. See \cite[Section
8]{LOT10b} for details. One consequence of this Koszul duality is the existence
of an $A_\infty$ morphism from $\sla(\slz)$ to the cobar resolution of
$\sla(\slz)$, inducing an isomorphism on homology. We will give some explicit
computations of this map as an application.

In addition to giving a more concrete understanding of Koszul duality in
$\sla(\slz)$, an explicit description of $\widehatit{CFAA}(\iz)$ can be useful
in various computations. In addition to $\widehatit{CFD}$ and $\widehatit{CFA}$,
there are also bimodule invariants $\widehatit{CFDD}, \widehatit{CFDA}$ and
$\widehatit{CFAA}$ associated to any 3-manifold with two boundaries, in
particular mapping cylinders of surface diffeomorphisms. In general, it is
easier to compute $\widehatit{CFD}$ and $\widehatit{CFDD}$, since it involves
counting simpler holomorphic curves, and there are known methods to exploit the
type $\mathit{DD}$ structure equations. These methods are used to compute
$\widehatit{CFDD}(\iz)$ in \cite{LOT10b, LOT10c}, and $\widehatit{CFDD}(\tau)$
for any arcslide $\tau$. With an explicit description of
$\widehatit{CFAA}(\iz)$, we can then obtain descriptions of type $A$ and type
$\mathit{DA}$ invariants with the same number of generators whenever the type
$D$ and type $\mathit{DD}$ invariants can be computed. This is used in
\cite{BZ2} to give a combinatorial construction and proof of invariance for
$\widehatit{HF}$ using bordered Floer theory.

The construction of the rank 1 model of $\widehatit{CFAA}(\iz)$ begins with a
previously known model $\mathcal{M}$ of $\widehatit{CFAA}(\iz)$. The chain
complex $M$ underlying $\mathcal{M}$ has far more generators than is necessary
for $\widehatit{CFAA}(\iz)$. The tool used to reduce the number of generators is
\emph{homological perturbation theory}. To use this theory, we find a smaller
chain complex $N$ that is homotopy equivalent to $M$. Since the theory is over
$\mathbb{F}_2$, we may take $N$ to be the homology of $M$. Homological
perturbation theory will construct an $A_\infty$-bimodule $\mathcal{N}$ homotopy
equivalent to $\mathcal{M}$, whose underlying chain complex is $N$.

To construct $\mathcal{N}$, we need the additional data that verifies the
homotopy equivalence between $M$ and $N$. That is, morphisms $f:M\to N$ and
$g:N\to M$ such that $g\circ f = \mathbb{I}_N$, and a homotopy map $H:M\to M$
such that $\mathbb{I}_M + f\circ g = d\circ H + H\circ d$. Both $f$ and $g$ will
become obvious after we describe the chain complex and its homology. So the
focus of this paper will be on constructing $H$ and verifying that it is indeed
a homotopy.

When $\slz$ is the (unique) genus 1 pointed matched circle, the size of the
chain complex $M$ is small enough that $H$ can be found directly. This is done
in \cite[Section 8.4]{LOT10c}. The computation here works for pointed matched
circles of any genus, and one can easily check that it agrees with the previous
computation in the genus 1 case.

We now describe plans for the rest of this paper. In Section \ref{sec:chaincx},
we will review some algebraic concepts, and describe the initial, larger model
$\mathcal{M}$ of $\widehatit{CFAA}(\iz)$. In Section \ref{sec:homotopy} we
describe the homotopy map $H$. In Section \ref{sec:proof} we verify that it
satisfies the homotopy equation. In Section \ref{sec:typeaa}, we apply
homological perturbation theory to describe the rank-1 model $\mathcal{N}$ of
$\widehatit{CFAA}(\iz)$, and prove the quasi-invertibility of
$\widehatit{CFDD}(\iz)$. Finally, in Section \ref{sec:exkoszul} we give an
application calculating some homology classes in the cobar resolution of
$\sla(\slz)$.

\subsection{Acknowledgements}

I would like to thank Peter Ozsv\'ath for suggesting this problem, and him and
Zolt\'an Szab\'o for many helpful discussions in the course of writing this
paper. Finally, I would like to thank the referee for pointing out various
improvements to the presentation in the paper.

\section{Algebraic Preliminaries}\label{sec:chaincx}

In this section we briefly review some algebraic concepts, and describe the
initial model $\mathcal{M}$ of the $A_\infty$-bimodule
$\widehatit{CFAA}(\iz)_{\sla(-\slz),\sla(\slz)}$.

We assume that the reader is familiar with pointed matched circles and the
dg-algebra $\sla(\slz)$ associated to a pointed matched circle $\slz$. In most
parts of this paper, we will have in mind some fixed pointed matched circle
$\slz$. Then, a generator of $\sla(\slz)$ will be represented by upward-veering
strands, and a generator of $\sla(-\slz)$ by downward-veering ones. Paired
horizontal strands in the generator will be shown using dashed lines. Later on,
when we are dealing solely with generators of $\sla(\slz)$, we will also omit
the direction markers on strands.

After fixing a pointed matched circle $\slz$, we will write $\sla$ and $\sla'$
for the dg-algebras $\sla(\slz)$ and $\sla(-\slz)$. For any element $a\in\sla$,
its corresponding element in $\sla'$ is denoted $\bar{a}$. In particular this
applies to idempotents. If $i\in\sla$ is an idempotent, we also define
$o(i)\in\sla$ to be the idempotent complementary to $i$.

Let $A$ be a dg-algebra over a ground ring $\mathbf{k}$ (which will be the
direct sum of copies of $\mathbb{F}_2$). Recall that a right $A_\infty$-module
$\mathcal{M}_A$ over $A$ consists of a module $M$ over $\mathbf{k}$, together
with multilinear maps:
\[ m_{1,i} : M \otimes A^{\otimes i} \to M, \] for all $i\ge 0$, where
$A^{\otimes i}$ denotes the tensor product of $i$ copies of $A$, and $A^{\otimes
  0}$ is just $\mathbf{k}$. All tensor products are implicitly taken over
$\mathbf{k}$. These maps satisfy the $A_\infty$ structure equation:

\begin{align} \label{eq:typeastreq}
  0 &= \sum_{i+j=n} m_{1,j}(m_{1,i}(\mathbf{x}, a_1, \dots, a_i), a_{i+1}, \dots, a_n) \nonumber \\
  & + \sum_{i=1}^n m_{1,n}(\mathbf{x}, a_1, \dots, da_i, \dots, a_n) \nonumber \\
  & + \sum_{i=1}^{n-1} m_{1,n-1}(\mathbf{x}, a_1, \dots, a_ia_{i+1}, \dots,
  a_n).
\end{align}

If we ignore those maps $m_{1,i}$ with $i>0$, we get a chain complex, which is
called the chain complex underlying $\mathcal{M}$.

A left type $D$ structure $^A\mathcal{N}$ over $A$ consists of a module $N$ over
$\mathbf{k}$, together with a map:
\[ \delta^1 : N \to A \otimes N, \] satisfying the type $D$ structure equation:

\begin{equation} \label{eq:typedstreq} (\mu_2\otimes \mathbb{I}_N) \circ
  (\mathbb{I}_A\otimes \delta^1) \circ \delta^1 + (\mu_1\otimes \mathbb{I}_N)
  \circ \delta^1 = 0.
\end{equation}
Here $\mu_1:A\to A$ and $\mu_2:A\otimes A\to A$ denote differential and
multiplication on the dg-algebra $A$ (in keeping with the notation for the more
general case, where $A$ is an $A_\infty$-algebra). Later on we will also call
type $D$ structures ``modules''.

In bordered Floer theory, the ground ring $\mathbf{k}$ is the direct sum of
copies of $\mathbb{F}_2$, one for each indecomposable idempotent in
$\sla(\slz)$. For each dg-algebra $\sla(\slz)$ and each $A_\infty$-module or
type $D$ structure, it is always possible to choose (often canonically) a set of
generators over $\mathbb{F}_2$, such that each generator $\mathbf{x}$ has an
indecomposable idempotent $i(\mathbf{x})$ satisfying $\mathbf{x} =
i(\mathbf{x})\mathbf{x}$. Intuitively, we can think of algebras and modules as
generated over $\mathbb{F}_2$, but each generator has an idempotent, such that
idempotents are required to match in algebra actions and in structure equations
like (\ref{eq:typeastreq}) and (\ref{eq:typedstreq}).

Given a right $A_\infty$-module $\mathcal{M}_A$ and a left type $D$ structure
$^A\mathcal{N}$, the box tensor product
$\mathcal{M}_A\boxtimes\tensor[^A]{\mathcal{N}}{}$ is a chain complex whose
underlying vector space is $M\otimes N$, and whose differential is given by:
\begin{equation} \label{eq:boxtensordef}
  \partial(\mathbf{x}\otimes\mathbf{y}) =
  \sum_{k=0}^\infty(m_{1,k}\otimes\mathbb{I}_N)
  (\mathbf{x}\otimes\delta^k(\mathbf{y})).
\end{equation}
Here $\delta^k:N\to A^{\otimes k}\otimes N$ is given by applying $\delta^1$
repeatedly on the $N$ factor $k$ times, while applying the identity map on all
$A$ factors at each step. The sum in (\ref{eq:boxtensordef}) is finite under
certain boundedness conditions on $\mathcal{M}_A$ and $^A\mathcal{N}$. In
bordered Floer theory, these boundedness conditions correspond to admissability
conditions on the Heegaard diagrams (see \cite[Sections 2.4 and 4.4]{LOT08}).

There are analogous definitions for bimodules, and the box tensor products
between them. They are given in detail in \cite{LOT10a}. Given dg-algebras $A$
and $A'$, a right $A_\infty$-bimodule $\mathcal{M}_{A,A'}$ over $A$ and $A'$ is
a module $M$ over $\mathbf{k}$ with structure maps
\[ m_{1,i,j} : M \otimes A^{\otimes i} \otimes A'^{\otimes j} \to M.\] A type
$\mathit{DA}$ structure (or bimodule) $^A\mathcal{N}_{A'}$ over $A$ and $A'$ is
a module $N$ over $\mathbf{k}$ with structure maps
\[ \delta^1_i : N \otimes A'^{\otimes i} \to A \otimes N.\] Finally, a type
$\mathit{DD}$ structure (or bimodule) $^{A,A'}\mathcal{N}$ over $A$ and $A'$ is
a module $N$ over $\mathbf{k}$ with structure maps
\[ \delta^1 : N \to A \otimes A' \otimes N.\] In each case, the structure maps
satisfy a structure equation analogous to Equations (\ref{eq:typeastreq}) and
(\ref{eq:typedstreq}). Note each generator of the bimodules has two idempotents,
one for each algebra action. Also, we used the notational convention of writing
each dg-algebra on the side it acts on, with superscripts indicating type $D$
action, and subscripts indicating $A_\infty$ (or type $A$) action.

For any module or bimodule $\mathcal{M}$, we denote its opposite by
$\overline{\mathcal{M}}$ (see \cite[Definition 2.2.31]{LOT10a}). Taking the
opposite switches the side of all algebra actions (or equivalently, replacing
action of an algebra $A$ by its opposite $A^{\mathrm{opp}}$).

In all examples of chain complexes and modules that we will encounter in this
paper, there is a canonical choice of generators of the underlying vector
space. In such cases, we can describe the differentials or structure maps as a
sum of \emph{arrows}. For chain complexes, the differential is a sum of arrows
$\mathbf{x}\to\mathbf{y}$, where $\mathbf{x}$ and $\mathbf{y}$ are generators of
the chain complex. The arrow maps $\mathbf{x}$ to $\mathbf{y}$, and all other
generators to zero. Similarly, the structure map of a left type $D$ module over
$A$ is a sum of arrows $\mathbf{x}\to a\otimes\mathbf{y}$, where $\mathbf{x}$
and $\mathbf{y}$ are generators of the module, and $a$ is a generator of
$A$. The structure map of a right $A_\infty$-module over $A$ is a sum of arrows
of the form $m_{1,i}(\mathbf{x},a_1,\dots,a_i)\to\mathbf{y}$, where $\mathbf{x}$
and $\mathbf{y}$ are generators of the module, and $a_1,\dots,a_i$ are
generators of $A$. This terminology extends in a straightforward way to the
various types of bimodules.

For dg-algebras $A$ and $B$ with ring of idempotents $\mathbf{k}$, we say a
bimodule $M$ (of type $\mathit{DD}$, $\mathit{DA}$, or $\mathit{AA}$) over $A$
and $B$ has \emph{rank-1} if its underlying module over $\mathbf{k}$ is free of
rank 1. There is a correspondence between rank-1 type $\mathit{DA}$ bimodules
$^BM_A$ with $\delta^1_1=0$ and $A_\infty$-morphisms $\phi:A\to B$ (\cite[Lemma
2.2.50]{LOT10a}). Given $\phi:A\to B$, the corresponding bimodule is denoted
$^B[\phi]_A$. It has type $\mathit{DA}$ actions $\delta^1_1=0$ and
\[ \delta^1_{k+1}(\mathbf{1}, a_1, \dots, a_k) = \phi(a_1, \dots, a_k) \otimes
\mathbf{1}. \]

We now describe in detail the initial model $\mathcal{M}$ of the
$A_\infty$-bimodule $\widehatit{CFAA}(\iz)$, using the formula from
\cite[Proposition 9.2]{LOT10a}:
\begin{eqnarray} \label{eq:typeaaeq} \widehatit{CFAA}(\iz)_{\sla,\sla'} &=&
  \mathrm{Mor}^{\sla}( \tensor*[^{\sla}_{\sla'}]{\widehatit{CFDD}(\iz)}{},
  \tensor[^\sla]{\mathbb{I}}{_\sla}) \nonumber \\
  &=& \mathrm{Mor}^{\sla}( \tensor[_{\sla'}]{{\sla'}}{_{\sla'}} \boxtimes
  \tensor[^{\sla,\sla'}]{\widehatit{CFDD}(\iz)}{},
  \tensor[^\sla]{\mathbb{I}}{_\sla}) \nonumber \\
  &=& \overline{\widehatit{CFDD}(\iz)}^{\sla',\sla} \boxtimes
  \tensor[_\sla]{\sla}{_\sla} \boxtimes
  \tensor[_{\sla'}]{\overline{\sla'}}{_{\sla'}},
\end{eqnarray}
where the second line expands the definition of
$\tensor*[^{\sla}_{\sla'}]{\widehatit{CFDD}(\iz)}{}$, and the third line uses
the identity $\mathrm{Mor}(M,N) = \overline{M} \boxtimes N$. We begin by
describing each of the three factors in the last line of (\ref{eq:typeaaeq}).

The type $\mathit{DD}$ structure
$\tensor[^{\sla,\sla'}]{\widehatit{CFDD}(\iz)}{}$ is computed in \cite{LOT10c}.
It is generated over $\mathbb{F}_2$ by the set of pairs of complementary
idempotents $i\otimes i'$, where $i\in\sla$ and
$i'=\overline{o(i)}\in\sla'$. The type $\mathit{DD}$ action on $i\otimes i'$ is
given by:

\begin{equation} \label{eq:ddop} \delta^1(i\otimes i') =
  \sum_{\xi\in\mathcal{C}, ia(\xi)=a(\xi)j,
    i'\overline{a(\xi)}=\overline{a(\xi)}j'} (a(\xi) \otimes \overline{a(\xi)})
  \otimes (j\otimes j').
\end{equation}

Here $\mathcal{C}$ is the set of chords on $\mathcal{Z}$, $a(\xi)\in\sla$ is the
algebra element formed by summing all ways of adding horizontal strands to
$\xi$, and $\overline{a(\xi)}$ is the corresponding element in $\sla'$. For
example, an arrow in the genus 1 case is:
\begin{equation} \label{eq:exddidarrow}
  \delta^1\left(\begin{smalldsd}\ldoublehor{1}{3}\rdoublehor{2}{4}\end{smalldsd}\right)
  \to
  \left(\begin{smallsd}\strandup{1}{2}\end{smallsd}\otimes\begin{smallsd}\stranddown{1}{2}\end{smallsd}\right)\otimes
\begin{smalldsd}\ldoublehor{2}{4}\rdoublehor{1}{3}\end{smalldsd},
\end{equation}
where
\[ i=\begin{smallsd}\doublehor{1}{3}\end{smallsd}~, ~~
j=\begin{smallsd}\doublehor{2}{4}\end{smallsd}~, ~~\mathrm{and}~
a(\xi)=\begin{smallsd}\strandup{1}{2}\end{smallsd}~. \]

Note that the generators of $\widehatit{CFDD}(\iz)$ are in one-to-one
correspondence with idempotents of $\sla$. So $\widehatit{CFDD}(\iz)$ has rank 1
as a module over $\mathbf{k}$.

On taking the opposite, the directions of the arrows are reversed while the
coefficients are kept the same (that is, an arrow $\mathbf{x} \to a\otimes
\mathbf{y}$ in $\mathcal{M}$ corresponds to the arrow $\mathbf{y}\to a\otimes
\mathbf{x}$ in $\overline{\mathcal{M}}$). For later convenience we will also
reverse the order of the two algebra actions, so that idempotents and algebra
elements in $\sla'$ come first. As an example, the arrow in
(\ref{eq:exddidarrow}) gives raise to the following arrow in
$\overline{\widehatit{CFDD}(\iz)}^{\sla',\sla}$ (which is $\delta^1(j'\otimes
j)\to (\overline{a(\xi)}\otimes a(\xi))\otimes (i'\otimes i)$):

\begin{equation} \label{eq:exdualddidarrow}
  \delta^1\left(\begin{smalldsd}\ldoublehor{1}{3}\rdoublehor{2}{4}\end{smalldsd}\right)
  \to
  \left(\begin{smallsd}\stranddown{1}{2}\end{smallsd}\otimes\begin{smallsd}\strandup{1}{2}\end{smallsd}\right)\otimes
  \begin{smalldsd}\ldoublehor{2}{4}\rdoublehor{1}{3}\end{smalldsd}.
\end{equation}

This is shown graphically in Figure \ref{fig:ddexample}.
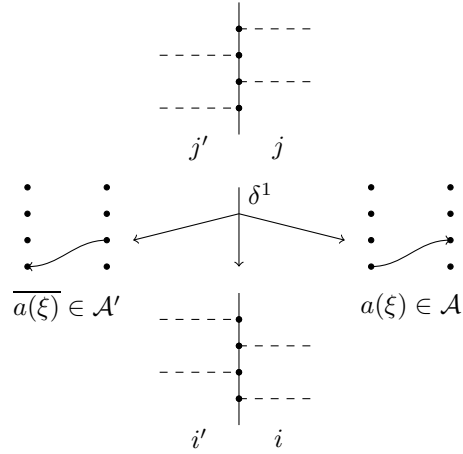
\begin{figure}[h!tb]
  \centering
  \begin{tikzpicture} [x=10pt,y=10pt]
    \begin{smallsubsd}{0}{11}\ldoublehor{1}{3}\rdoublehor{2}{4}\end{smallsubsd}
    \begin{smallsubsd}{0}{0}\ldoublehor{2}{4}\rdoublehor{1}{3}\end{smallsubsd}
    \begin{smallsubssd}{-5}{5}\stranddown{1}{2}\end{smallsubssd}
    \begin{smallsubssd}{8}{5}\strandup{1}{2}\end{smallsubssd}
    \draw [->] (3,9) to (3,6);
    \draw [->] (3,8) to (-1,7);
    \draw [->] (3,8) to (7,7);
    \draw (3,8) node[above right] {$\delta^1$};
    \draw (-3.5,4.5) node {$\overline{a(\xi)}\in\sla'$};
    \draw (9.5,4.5) node {$a(\xi)\in\sla$};
    \draw (1.5,10.5) node {$j'$};
    \draw (4.5,10.5) node {$j$};
    \draw (1.5,-0.5) node {$i'$};
    \draw (4.5,-0.5) node {$i$};
  \end{tikzpicture}
  \caption{An example of an arrow in
    $\overline{\widehatit{CFDD}(\iz)}^{\sla',\sla}$. Note strands in $\sla$ go
    upward and strands in $\sla'$ go downward. For later convenience, we put
    $\sla'$ on the left and $\sla$ on the right, but keep in mind that both are
    right type $D$ actions.}
  \label{fig:ddexample}
\end{figure}

The bimodule $_\sla\sla_\sla$ is a left-right $A_\infty$-bimodule with the same
generators as $\sla$. The $A_\infty$-bimodule actions simply come from the
dg-algebra actions on $\sla$. That is:
\begin{align}
  m_{0,1,0}(a) &\to b \textrm{, for each generator } b \textrm{ in } da.
  \nonumber \\
  m_{1,1,0}(b; a) &\to ba \textrm{, whenever } ba\neq 0. \nonumber \\
  m_{0,1,1}(a; b) &\to ab \textrm{, whenever } ab\neq 0. \nonumber
\end{align}

Here $m_{i,1,k}$ is the part of the $A_\infty$-bimodule action with $i$ algebra
inputs on the left and $k$ algebra inputs on the right. We write the algebra and
module inputs in order, omitting any part with zero elements.

The bimodule $_{\sla'}{\sla'}_{\sla'}$ is defined the same way, based on the
dg-algebra $\sla'$. By taking opposites, the directions of the arrows are
reversed. The algebra coefficients stay the same, but acting on the opposite
side. So the actions on $_{\sla'}\overline{\sla'}_{\sla'}$ can be written as:
\begin{eqnarray}
  m_{0,1,0}(b) &\to& a \textrm{, for each generator } b \textrm{ in } da.
  \nonumber \\
  m_{0,1,1}(ba; b) &\to& a \textrm{, whenever } ba\neq 0. \nonumber \\
  m_{1,1,0}(b; ab) &\to& a \textrm{, whenever } ab\neq 0. \nonumber
\end{eqnarray}

We can now describe the bimodule $\mathcal{M}$ in (\ref{eq:typeaaeq}). The
underlying vector space is generated by triples $[a', i'\otimes i, a]$ such that
$i'=\overline{o(i)}$, the idempotent $i'$ agrees with the right idempotent of
$a'$, and $i$ agrees with the left idempotent of $a$. Since $i'$ and $i$ are
determined by $a'$ and $a$, we will omit them and simply write the pair $[a',
a]$. The condition on idempotents becomes that the right idempotent of $a'$ is
complementary to the left idempotent of $a$ (using the identification between
idempotents of $\sla$ and $\sla'$).

The differentials in $_{\sla'}\overline{\sla'}_{\sla'}$ and $_\sla\sla_\sla$
give rise to the following arrows in $\mathcal{M}$:
\begin{eqnarray}
  {[}a', a] &\to& {[}a', b] \textrm{, for each generator } b \textrm{ in } da.
  \nonumber \\
  {[}b', a] &\to& {[}a', a] \textrm{, for each generator } b' \textrm{ in }
  da'. \nonumber
\end{eqnarray}

There is a third type of arrows, coming from the type $\mathit{DD}$ action on
$\overline{\widehatit{CFDD}(\iz)}^{\sla',\sla}$, combined with the $m_{1,1,0}$
actions on $_\sla\sla_\sla$ and
$\tensor[_{\sla'}]{\overline{\sla'}}{_{\sla'}}$. The arrows are of the form:
\[ [a'\overline{a(\xi)}, a] \to [a', a(\xi)a], \] for each chord $\xi$, and
$a\in\sla, a'\in\sla'$ such that $a(\xi)a\neq 0$ and $a'\overline{a(\xi)}\neq
0$.

As an example, we show in Figure \ref{fig:arrowformexample} how one such arrow
follows from the definition of $\cdot\boxtimes\cdot$. This example involves the
type $\mathit{DD}$ arrow in Figure \ref{fig:ddexample}.

\begin{figure}[h!tb]
  \centering
  \begin{tikzpicture} [x=10pt,y=10pt]
    \begin{smallsubsd}{0}{10}\ldoublehor{1}{3}\rdoublehor{2}{4}\end{smallsubsd}
    \begin{smallsubsd}{0}{0}\ldoublehor{2}{4}\rdoublehor{1}{3}\end{smallsubsd}
    \begin{smallsubssd}{-5}{5}\stranddown{1}{2}\end{smallsubssd}
    \begin{smallsubssd}{-10}{10}\stranddown{1}{3}\end{smallsubssd}
    \begin{smallsubssd}{-10}{0}\stranddown{2}{3}\end{smallsubssd}
    \begin{smallsubssd}{8}{5}\strandup{1}{2}\end{smallsubssd}
    \begin{smallsubssd}{13}{10}\strandup{2}{3}\end{smallsubssd}
    \begin{smallsubssd}{13}{0}\strandup{1}{3}\end{smallsubssd}
    \draw [->] (3,9) to (3,6);
    \draw [->] (3,8) to (-1,7);
    \draw [->] (3,8) to (7,7);
    \draw [->] (-8.5,9) to (-8.5,6);
    \draw (-6,8) to (-8.5,7);
    \draw [->] (14.5,9) to (14.5,6);
    \draw (12,8) to (14.5,7);
    \draw (3,8) node[above right] {$\delta^1$};
    \draw (-8.5,7) node[left] {$m_{1,1,0}$};
    \draw (14.5,7) node[right] {$m_{1,1,0}$};
    \draw (-3.5,5) node {$\sla'$};
    \draw (9.5,5) node {$\sla$};
  \end{tikzpicture}
  \caption{Example of the formation of an arrow in $\mathcal{M}$. From left to
    right, the three arrows composing it come from the $A_\infty$-bimodule
    action $m_{1,1,0}$ on $\overline{\sla'}$, the type $\mathit{DD}$ action
    $\delta^1$ on $\overline{\widehatit{CFDD}(\iz)}$, and the
    $A_\infty$-bimodule action $m_{1,1,0}$ on $\sla$.}
  \label{fig:arrowformexample}
\end{figure}
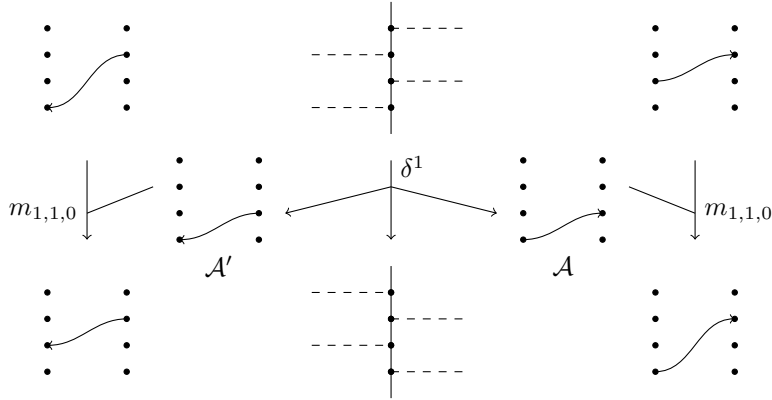

The right $A_\infty$ actions on $\mathcal{M}$ are simply the ones inherited from
those on $\sla'$ and $\sla$ (this is because $_\sla\sla_\sla$ and
$_{\sla'}\overline{\sla'}_{\sla'}$ have no $A_\infty$ actions with both left and
right algebra inputs). They are given by the arrows:
\begin{eqnarray}
  m_{1,1,0}({[}a', a]; b) \to {[}a', ab] \textrm{, whenever } ab\neq 0.
  \nonumber \\
  m_{1,0,1}({[}b'a', a]; b') \to {[}a', a] \textrm{, whenever } b'a'\neq 0.
  \nonumber
\end{eqnarray}
Here $m_{1,j,k}$ is the part of the $A_\infty$-bimodule action with $j$ inputs
from $\sla$ and $k$ inputs from $\sla'$. The first of the two equations comes
from the action on $\sla$, and the second from the action on $\sla'$.

For simplicity of the discussion later, we will write everything in terms of
elements of $\sla$, so that the pair $[a', a]$ is written as the pair $[a_1,
a_2]$, where $a_1=\overline{a'}$ and $a_2=a$. Translating the differential and
the condition on idempotents, we arrive at the following statement:

\begin{proposition} The $A_\infty$-bimodule $\mathcal{M}$ is generated by pairs
  $[a_1, a_2]$, where $a_1$ and $a_2$ are generators of $\sla$, such that the
  left idempotent of $a_1$ is complementary to the left idempotent of
  $a_2$. There are three types of arrows in the differential:
  \begin{eqnarray}
    {[}a_1, a_2] &\to& {[}a_1, b] \textrm{, for each generator } b \textrm{ in } da_2 \nonumber \\
    {[}b, a_2] &\to& {[}a_1, a_2] \textrm{, for each generator } b \textrm{ in } da_1 \nonumber \\
    {[}a(\xi)c, a] &\to& {[}c, a(\xi)a] \textrm{, for each chord } \xi \textrm { and }
    a, c\in\sla \textrm{ such that } a(\xi)c\neq 0 \textrm{ and } a(\xi)a\neq 0. \nonumber
  \end{eqnarray}

  The $A_\infty$-bimodule action consists of the following arrows:
  \begin{eqnarray}
    m_{1,1,0}({[}a_1, a_2]; a) &\to& {[}a_1, a_2a] \textrm{, whenever } a_2a\neq 0 \nonumber \\
    m_{1,0,1}({[}a_1\overline{a'}, a_2]; a') &\to& {[}a_1, a_2]
    \textrm{, whenever } a_1\overline{a'}\neq 0 \nonumber
  \end{eqnarray}
\end{proposition}

In the following sections, we will draw a generator of $\mathcal{M}$ by drawing
$a_1$ and $a_2$ side by side. All strands are then going upward, so we omit
directions on strands. Using this convention, examples of the three types of
arrows in the differential are shown in Figure \ref{fig:exdifferentials}.

\begin{figure}[h!tb]
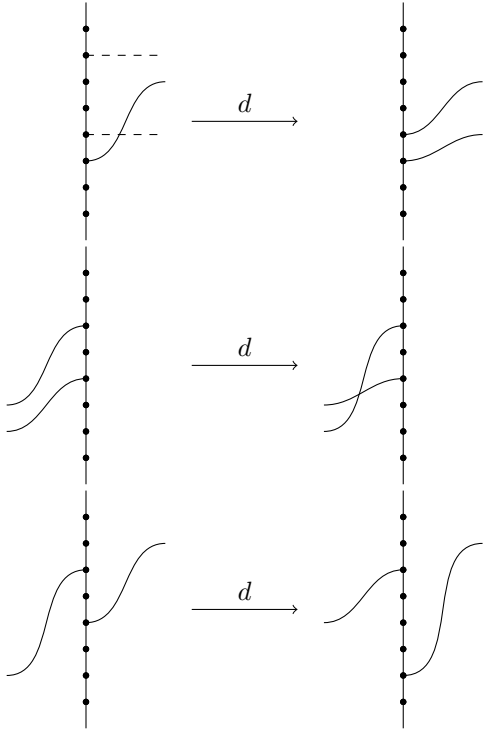

  \mapdiag{0}{0}{$d$}
  {\rstrand{3}{6}\rdoublehor{4}{7}}
  {\rstrand{3}{4}\rstrand{4}{6}}
  \\ \\
  \mapdiag{0}{0}{$d$}
  {\lstrand{2}{4}\lstrand{3}{6}}
  {\lstrand{2}{6}\lstrand{3}{4}}
  \\ \\
  \mapdiag{0}{0}{$d$}
  {\lstrand{2}{6}\rstrand{4}{7}}
  {\lstrand{4}{6}\rstrand{2}{7}}
  \caption{Examples of arrows in the chain complex $M$ underlying
    $\mathcal{M}$.}
  \label{fig:exdifferentials}
\end{figure}

\section{Description of the homotopy}\label{sec:homotopy}

Let $(M, d)$ be the chain complex underlying $\mathcal{M}$, described in the
previous section. There is a distinguished set of generators of $M$, consisting
of $[a_1,a_2]$ where both $a_1$ and $a_2$ are idempotents (which are necessarily
complementary).  By abuse of notation we will also call these generators
\emph{idempotents}.  It is clear that there are no arrows in the differential in
or out of these generators. Let $(N,d'=0)$ be the subcomplex of $M$ generated by
the idempotents. We will show that $N$ is the homology of $M$, so that $N$ and
$M$ are homotopy equivalent. There are obvious chain maps $f:N\to M$ and $g:M\to
N$, where $f$ is the inclusion map and $g$ is the map sending the idempotents to
themselves and all other generators to zero. It remains to find a homotopy
$H:M\to M$ satisfying
\[ d\circ H + H\circ d = \mathbb{I}_M + f\circ g. \]

We set $H(\mathbf{x})=0$ for all idempotents $\mathbf{x}$. Then the above
equation holds trivially on idempotents. On the other generators, we have
$f\circ g=0$, so the equation reduces to $d\circ H + H\circ d=\mathbb{I}_M$.

We now describe $H$ on the non-idempotents, and verify the required relation in
the next section. From now on all generators that we mention are assumed to be a
non-idempotent. We will call the arrows in the differential of $M$ the
$d$-arrows, and the arrows in the map $H$ the $H$-arrows.

For a generator $[a_1, a_2]$ in $C$, we define the multiplicity
$\mathrm{mult}([a_1, a_2])\in H_1(Z\setminus z, \mathbf{a})$ to be the sum of
multiplicities of $a_1$ and $a_2$. Note this is invariant under the
differential. So the chain complex splits into disjoint parts according to
multiplicity, and to specify the homotopy, it is sufficient to do so on each
part. We say an element $[a_1, a_2]$ has multiplicity one if its multiplicity on
each component of $Z\setminus\mathbf{a}$ is at most one.  The construction of
$H$ for generators of multiplicity one is different from that for other
generators.

\subsection{Multiplicity one}
We begin with generators of multiplicity one. Let $S$ be the set of components
of $Z\setminus\mathbf{a}$ not containing the basepoint $z$ (that is, the
generators of $H_1(Z\setminus z, \mathbf{a})$). There is an ordering $<_\slz$ on
$S$, depending on the pointed matched circle $\slz$, given as follows.  Attach
handles to $Z$ according to the matching on $\mathbf{a}$. Then, traverse the
boundary of $Z$ with handles attached, starting from the top (with $Z$ oriented
so that strands in $\sla(\slz)$ go upwards). Order $S$ according to when we
encounter each element.  Note we will encounter each element of $S$ exactly once
before reaching the bottom due to the condition that must be satisfied for the
matching on a pointed matched circle. An example is shown in Figure
\ref{fig:extraverse}.

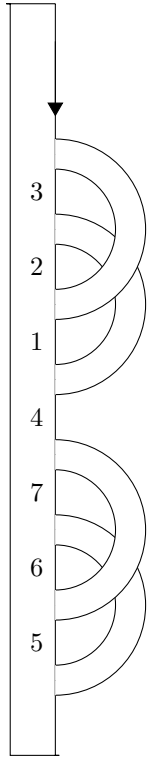
\begin{figure}[h!tb]
  \centering
  \begin{tikzpicture}
    \draw [-] (0,1) to (0,-9) to [out=0,in=0] (-0.6,-9) to (-0.6,1) to [out=180,in=180] (0,1);
    \filldraw[fill=white,draw=black] (0,-3.8) arc (-90:90:0.8) -- (0,-1.8) arc (90:-90:1.2) -- (0,-3.8);
    \filldraw[fill=white,draw=black] (0,-2.8) arc (-90:90:0.8) -- (0,-0.8) arc (90:-90:1.2) -- (0,-2.8);
    \filldraw[fill=white,draw=black] (0,-7.8) arc (-90:90:0.8) -- (0,-5.8) arc (90:-90:1.2) -- (0,-7.8);
    \filldraw[fill=white,draw=black] (0,-6.8) arc (-90:90:0.8) -- (0,-4.8) arc (90:-90:1.2) -- (0,-6.8);
    \foreach \x in {-8,-7,...,-1} {
      \draw [white] (0,\x) -- +(0,-0.2);
      \draw [white] (0,\x) -- +(0,0.2);
    }
    \draw (0,-1.5) node[left=1pt] {3};
    \draw (0,-2.5) node[left=1pt] {2};
    \draw (0,-3.5) node[left=1pt] {1};
    \draw (0,-4.5) node[left=1pt] {4};
    \draw (0,-5.5) node[left=1pt] {7};
    \draw (0,-6.5) node[left=1pt] {6};
    \draw (0,-7.5) node[left=1pt] {5};
    \draw [->,>=triangle 60] (0,+0.5) -- (0,-0.5);
  \end{tikzpicture}
  \caption{Example of traversal on the split pointed matched circle with genus
    2. The numbers denote the ordering $<_\slz$ on $S$. Note strands in
    $\sla(\slz)$ go upwards.}
  \label{fig:extraverse}
\end{figure}

Let $(p,p+1)$ be an element of $S$ (with $p+1$ above $p$). The segment
immediately before $(p,p+1)$ in the ordering $<_\slz$ can be found as
follows. Let $q$ be the point paired with $p+1$ in $\slz$. If $q$ is the topmost
point of $\mathbf{a}\subset Z$, then $(p,p+1)$ is the initial segment.
Otherwise $(q,q+1)$ is the segment prior to $(p,p+1)$. See Figure
\ref{fig:exprior} for a demonstration.

\begin{figure}[h!tb]
  \centering
  \begin{tikzpicture}[x=20pt,y=20pt]
    \draw (0,0) to (0,5);
    \draw [ultra thick] (0,0) to (0,1);
    \draw (0,1) to [out=180,in=270] (-1.5,2.5) to [out=90,in=180] (0,4);
    \draw [->] (-0.3,5) to (-0.3,4.3) to [out=180,in=90] (-1.8,2.5) to [out=270,in=180] (-0.3,0.7) to (-0.3,0);
    \draw (0,0) node[right] {$p$};
    \draw (0,1) node[right] {$p+1$};
    \draw (0,4) node[right] {$q$};
    \draw (1,0.5) node[right] {key segment};
    \draw (1,4.5) node[right] {empty};
  \end{tikzpicture}
  \quad\quad
  \begin{tikzpicture}[x=20pt,y=20pt]
    \draw (0,0) to (0,5);
    \draw [ultra thick] (0,3) to (0,4);
    \draw (0,1) to [out=180,in=270] (-1.5,2.5) to [out=90,in=180] (0,4);
    \draw [->] (-0.3,2) to (-0.3,1.3) to [out=180,in=270] (-1.2,2.5) to [out=90,in=180] (-0.3,3.7) to (-0.3,3);
    \draw (0,1) node[right] {$q$};
    \draw (0,3) node[right] {$p$};
    \draw (0,4) node[right] {$p+1$};
    \draw (1,1.5) node[right] {empty};
    \draw (1,3.5) node[right] {key segment};
  \end{tikzpicture}
  \caption{Identifying the segment prior to $(p,p+1)$. In the first case the
    segment $(q,q+1)$ may in fact be the component of $Z$ containing the
    basepoint.}
  \label{fig:exprior}
\end{figure}
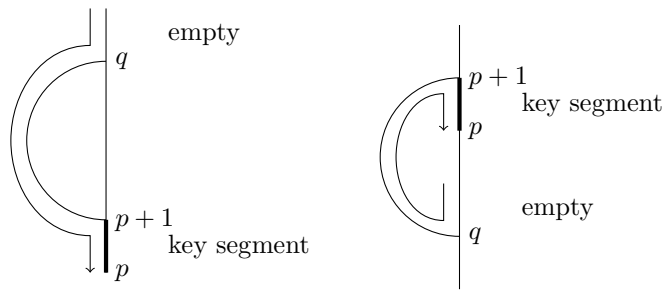

Given a generator $\mathbf{x}=[a_1, a_2]$ with multiplicity one, we define the
$\emph{key segment}$ of $\mathbf{x}$ to be the first segment in $S$, according
to the ordering $<_\mathcal{Z}$, at which $\mathrm{mult}(\mathbf{x})$ is
one. Let $(p,p+1)$ be the key segment, and $q$ be the point paired with
$p+1$. Define the pair $\{p+1, q\}$ to be the $\emph{key pair}$ of $\mathbf{x}$.
The main property that results from this construction is that while $\mathbf{x}$
has multiplicity 1 at $(p,p+1)$, it must have multiplicity zero at $(q,q+1)$
(one possibility is that $q$ is the topmost point).

In the following, given a generator $[a_1, a_2]$, we say that the key segment is
occupied on the left if it is covered by $a_1$, and on the right if it is
covered by $a_2$. Since the left idempotents of $a_1$ and $a_2$ are
complementary, one of them contains the key pair. We say the key pair is
occupied on the left if it is contained in the left idempotent of $a_1$, and on
the right if it is contained in the left idempotent of $a_2$. We also
distinguish whether the algebra element realizes the key pair as a
double-horizontal or as the start of a non-horizontal (moving) strand. Note in
the latter case, since the multiplicity at $(q,q+1)$ is zero, the strand must
start at $p+1$.

The above classification divides all (non-idempotent) generators into eight
types. We will define $H(\mathbf{x})$ based on the type of $\mathbf{x}$. 

In the following, let $a\to b$ denote the strand starting at $a$ and ending at
$b$ (so we always have $a<b$). By \emph{moving the strand $a\to b$ to the left},
we mean that starting from $[a_1, a_2]$, construct a new generator $[a_1',
a_2']$, with the strands in $a_2'$ obtained from that in $a_2$ by either
removing the strand $a\to b$, or by factoring $a\to b$ from the end of a longer
strand in $a_2$; and with the strands in $a_1'$ obtained from that in $a_1$ by
adding the strand $a\to b$. If rather than adding the strand $a\to b$, we wish
to attach it to an existing strand in $a_1$, we will state it
explicitly. Similarly we have the notion of moving the strand $a\to b$ to the
right. Sometimes we will make two moves at the same time (with the intermediate
state possibly not valid strand diagrams). The location of double horizontals on
$a_1'$ and $a_2'$ is usually clear (noting that the left idempotents of $a_1'$
and $a_2'$ must be complementary). We will clarify it when it is ambiguous.

In all figures illustrating multiplicity one cases, we will use a dotted
parenthesis to denote the key pair. We will also show $q$ to be above rather
than below $p+1$, so the key segment is just below the dotted parenthesis, but
the definition of $H$ is the same in both cases.

For exactly four types of $\mathbf{x}$, we have $H(\mathbf{x})\neq 0$. The
values of $H(\mathbf{x})$ in these four types are as follows (see Figure
\ref{fig:all-m1-homotopy}):

\begin{enumerate}
\item If the key segment is occupied on the left and the key pair is occupied on
  the left as a double-horizontal, then there is a strand $a\to b$ on the left
  with $a<p+1<b$. Set $H$ to resolve the crossing involving the horizontal
  strand at $p+1$.

\item If the key segment is occupied on the right and the key pair is occupied
  on the right by a moving strand, then there must be strands $i\to p+1$ and
  $p+1\to j$ on the right. Set $H$ to replace these two strands with the strand
  $i\to j$ and the double-horizontal at $\{p+1,q\}$.

  If there is a strand ending at $q$ on the left, then there is an additional
  term in $H$, moving the strand ending at $q$ to the right and the strand $i\to
  p+1$ to the left.

\item If the key segment is occupied on the right by a strand $i\to j$, and the
  key pair is occupied on the left as a double-horizontal, then $H$ factors the
  strand $i\to p+1$ from the right and moves it to the left.

\item If the key segment is occupied on the right, and the key pair is occupied
  on the left by a moving strand, then there must be a strand $p+1\to j$ on the
  left and a strand $i\to p+1$ on the right. Set $H$ to move the strand $i\to
  p+1$ to the left attaching it to the $p+1\to j$ strand, and leaving a double
  horizontal at $\{p+1, q\}$ at right.

  There are two special cases: first, if there is a strand $j\to q$ with $j\neq
  p+1$ on the left, then $H$ contains an additional term moving $j\to q$ to the
  right and $i\to p+1$ to the left.

  Second, if there is a strand $p+1\to q$ on the left (possible only if
  $q>p+1$), then $H$ contains an additional term moving $p+1\to q$ to the right
  and $i\to p+1$ to the left.
\end{enumerate}

\begin{figure}[h!tb]
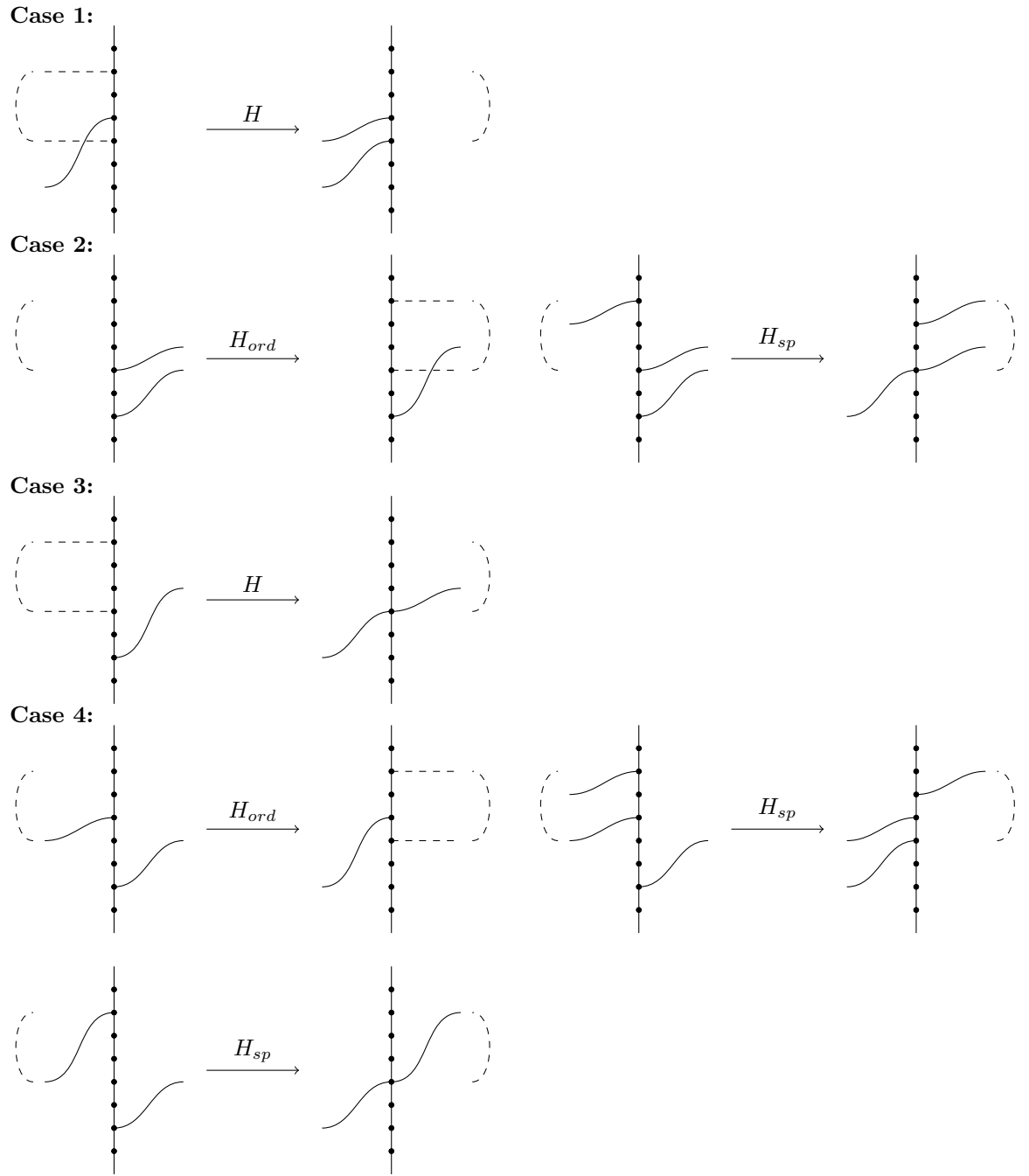

  \textbf{Case 1:} \\
  \mapdiag{4}{7}{$H$}
  {\lstrand{2}{5}\ldoublehor{4}{7}}
  {\lstrand{2}{4}\lstrand{4}{5}}
  \\
  \textbf{Case 2:} \\
  \mapdiag{4}{7}{$H_{ord}$}
  {\rstrand{2}{4}\rstrand{4}{5}}
  {\rstrand{2}{5}\rdoublehor{4}{7}}
  \quad
  \mapdiag{4}{7}{$H_{sp}$}
  {\lstrand{6}{7}\rstrand{2}{4}\rstrand{4}{5}}
  {\lstrand{2}{4}\rstrand{4}{5}\rstrand{6}{7}}
  \\
  \textbf{Case 3:} \\
  \mapdiag{4}{7}{$H$}
  {\ldoublehor{4}{7}\rstrand{2}{5}}
  {\lstrand{2}{4}\rstrand{4}{5}}
  \\
  \textbf{Case 4:} \\
  \mapdiag{4}{7}{$H_{ord}$}
  {\lstrand{4}{5}\rstrand{2}{4}}
  {\lstrand{2}{5}\rdoublehor{4}{7}}
  \quad
  \mapdiag{4}{7}{$H_{sp}$}
  {\lstrand{4}{5}\lstrand{6}{7}\rstrand{2}{4}}
  {\lstrand{2}{4}\lstrand{4}{5}\rstrand{6}{7}}
  \\ \\
  \mapdiag{4}{7}{$H_{sp}$}
  {\lstrand{4}{7}\rstrand{2}{4}}
  {\lstrand{2}{4}\rstrand{4}{7}}
  \caption{Diagrams for the homotopy map in multiplicity-one cases.}
  \label{fig:all-m1-homotopy}
\end{figure}

The overall picture is as follows: we partition all generators of multiplicity
one into ordered pairs, such that for each ordered pair $(\mathbf{x_i},
\mathbf{y_i})$, there is a $d$-arrow $\mathbf{x_i}\to\mathbf{y_i}$. The part of
$H$ not including the three special cases maps each $\mathbf{y_i}$ to
$\mathbf{x_i}$ and $\mathbf{x_i}$ to zero. In verifying the relation $d\circ
H+H\circ d=\mathbb{I}_M$, the compositions $d\circ
H:\mathbf{y_i}\to\mathbf{x_i}\to\mathbf{y_i}$ and $H\circ
d:\mathbf{x_i}\to\mathbf{y_i}\to\mathbf{x_i}$ account for the identity. If
$(\mathbf{x_i}, \mathbf{y_i})$ is an ordered pair, we say $\mathbf{x_i}$ is on
the \emph{$d$-side} and $\mathbf{y_i}$ is on the \emph{$H$-side}. The three
special cases in $H$ are additional maps from the $H$-side to the $d$-side (that
is, mapping $\mathbf{y_i}$ to $\mathbf{x_j}$ for some $i\neq j$). Intuitively,
the part of $H$ mapping each $\mathbf{y_i}$ to $\mathbf{x_i}$ performs the
inverse of a $d$-arrow around $p+1$. The special arrows perform a different move
around $p+1$, along with moving a strand ending at $q$ to the right.

Whether a generator is on the $d$-side or the $H$-side depends solely on the
type of the generator. The eight types are summarized in Table
\ref{table:summary}. The numbers after $d$ or $H$ indicate under which case they
will be covered in the proof in the next section, and if labeled $H$, also the
case of $H$ in the above description.

\begin{table}[ht]
  \centering
  \begin{tabular}{c|cccc}\hline
    \backslashbox{key seg.}{key pair} &
    left, double-hor. & right, double-hor. & left, moving & right, moving
    \\ \hline
    left & $H$-1 & $d$-7 & $d$-5 & $d$-8 \\
    right & $H$-3 & $d$-6 & $H$-4 & $H$-2 \\
  \end{tabular}
  \caption{Summary of cases for multiplicity one generators.}
  \label{table:summary}
\end{table}

It would be simpler if there were no special cases in $H$. That is, if the only
$H$-arrows were those from the $H$-side to the $d$-side of the same
pair. However, this is impossible as demonstrated in Figures
\ref{fig:specialhomotopy} and \ref{fig:specialhomotopy2}.  Each figure shows
four generators $\mathbf{x,y,z,w}$ such that $d\mathbf{x}=\mathbf{y}+\mathbf{w}$
and $d\mathbf{z}=\mathbf{w}$, with no other $d$-arrows involving these
generators. This forces the homotopy to be $H\mathbf{y}=\mathbf{x}+\mathbf{z}$
and $H\mathbf{w}=\mathbf{z}$.

\begin{figure} [h!tb]
  \subfloat[Case 4, second special $H$] {
    \begin{tikzpicture} [x=10pt,y=10pt]
      \begin{smallsubsd}{0}{10}\lstrand{1}{4}\rdoublehor{2}{4}\end{smallsubsd}
      \begin{smallsubsd}{12}{10}\lstrand{1}{2}\rstrand{2}{4}\end{smallsubsd}
      \begin{smallsubsd}{0}{0}\lstrand{2}{4}\rstrand{1}{2}\end{smallsubsd}
      \begin{smallsubsd}{12}{0}\ldoublehor{2}{4}\rstrand{1}{4}\end{smallsubsd}
      \draw [->] (3,9) to node [right] {$d$} (3,6);
      \draw [->] (15,9) to node [right] {$d$} (15,6);
      \draw [->] (7,9) to node [above right] {$d$} (11,6);
    \end{tikzpicture}
  }
  \caption{Examples of why special cases of homotopy are needed, part 1.}
  \label{fig:specialhomotopy}
\end{figure}

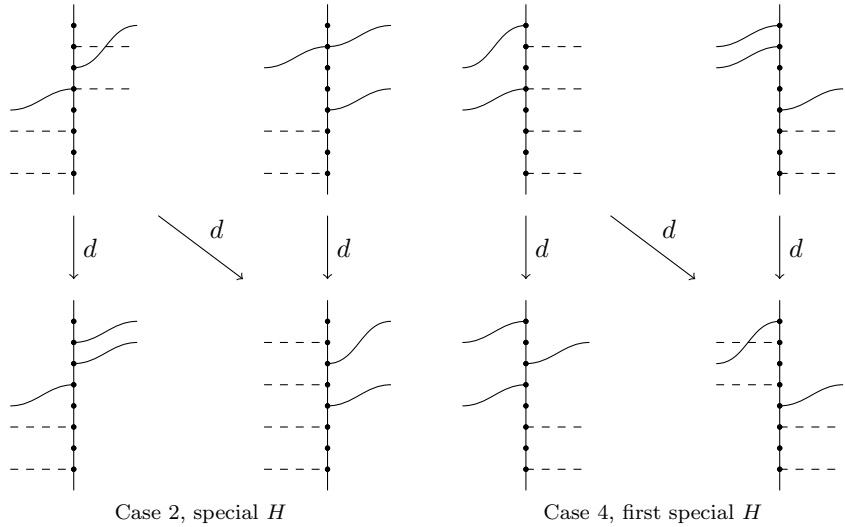
\begin{figure} [h!tb]
  \subfloat[Case 2, special $H$] {
    \begin{tikzpicture} [x=8pt,y=8pt]
      \begin{subsd}{0}{14}\lstrand{4}{5}\ldoublehor{1}{3}\rstrand{6}{8}\rdoublehor{5}{7}\end{subsd}
      \begin{subsd}{12}{14}\lstrand{6}{7}\ldoublehor{1}{3}\rstrand{4}{5}\rstrand{7}{8}\end{subsd}
      \begin{subsd}{0}{0}\lstrand{4}{5}\ldoublehor{1}{3}\rstrand{6}{7}\rstrand{7}{8}\end{subsd}
      \begin{subsd}{12}{0}\ldoublehor{1}{3}\ldoublehor{5}{7}\rstrand{4}{5}\rstrand{6}{8}\end{subsd}
      \draw [->] (3,13) to node [right] {$d$} (3,10);
      \draw [->] (15,13) to node [right] {$d$} (15,10);
      \draw [->] (7,13) to node [above right] {$d$} (11,10);
    \end{tikzpicture}
  }
  \quad\quad
  \subfloat[Case 4, first special $H$] {
    \begin{tikzpicture} [x=8pt,y=8pt]
      \begin{subsd}{0}{14}\lstrand{4}{5}\lstrand{6}{8}\rdoublehor{5}{7}\rdoublehor{1}{3}\end{subsd}
      \begin{subsd}{12}{14}\lstrand{6}{7}\lstrand{7}{8}\rstrand{4}{5}\rdoublehor{1}{3}\end{subsd}
      \begin{subsd}{0}{0}\lstrand{4}{5}\lstrand{7}{8}\rstrand{6}{7}\rdoublehor{1}{3}\end{subsd}
      \begin{subsd}{12}{0}\lstrand{6}{8}\ldoublehor{5}{7}\rstrand{4}{5}\rdoublehor{1}{3}\end{subsd}
      \draw [->] (3,13) to node [right] {$d$} (3,10);
      \draw [->] (15,13) to node [right] {$d$} (15,10);
      \draw [->] (7,13) to node [above right] {$d$} (11,10);
    \end{tikzpicture}
  }
  \caption{Examples of why special cases of homotopy are needed, part 2.}
  \label{fig:specialhomotopy2}
\end{figure}

\subsection{Other generators}
We now consider generators that do not satisfy the multiplicity one condition.
Let $\mathbf{x}=[a_1, a_2]$ be such a generator. This means
$\mathrm{mult}(\mathbf{x})$ is greater than 1 at some segment in $S$. Let $[i,
i+1]$ be the bottom-most segment with multiplicity greater than 1 (note we are
no longer using the ordering $<_\slz$ on $S$). Since two strands cannot start
from the same point $i$, this segment must have multiplicity 2, and the segment
$[i-1, i]$ must have multiplicity 1.  There is a unique strand starting at $i$
(call it the $i$ strand) and a unique strand that covers $[i-1, i]$. Let $j$ be
the starting point of the strand covering $[i-1,i]$, and call that strand the
$j$ strand. Note $i$, but not $j$, is unchanged in any $d$-arrow.

The definition of $H$ consists of the following cases (see Figure
\ref{fig:all-notm1-homotopy}):
\begin{enumerate}
\item If both the $i$ strand and the $j$ strand are on the left, and they cross
  each other, the $H$-arrow uncrosses the two strands.
\item If both the $i$ strand and the $j$ strand are on the right, and they do
  not cross, the $H$-arrow crosses the two strands.
\item If the $i$ strand is on the left and the $j$ strand is on the right, the
  $H$-arrow factors the strand $j\to i$ from the right and moves it to the left
  attaching to the $i$ strand.
\end{enumerate}

\begin{figure} [h!tb]
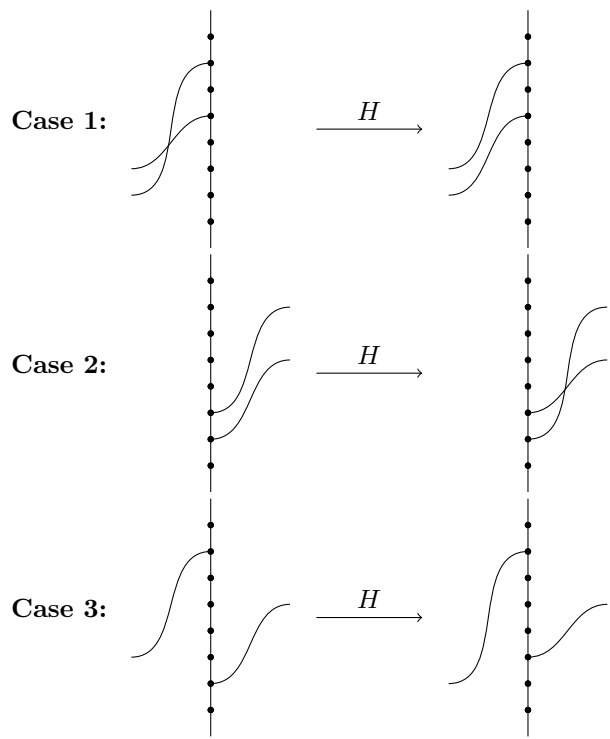

  \textbf{Case 1:} 
  \mapdiag{0}{0}{$H$}
  {\lstrand{2}{7}\lstrand{3}{5}}
  {\lstrand{2}{5}\lstrand{3}{7}}
  \\ \\
  \textbf{Case 2:} 
  \mapdiag{0}{0}{$H$}
  {\rstrand{2}{5}\rstrand{3}{7}}
  {\rstrand{2}{7}\rstrand{3}{5}}
  \\ \\
  \textbf{Case 3:} 
  \mapdiag{0}{0}{$H$}
  {\lstrand{3}{7}\rstrand{2}{5}}
  {\lstrand{2}{7}\rstrand{3}{5}}
  \caption{Diagrams for the homotopy map in the non-multiplicity-one cases.}
  \label{fig:all-notm1-homotopy}
\end{figure}

On all other generators $H$ is zero. The cases are summarized in Table
\ref{table:summary2}.  There is again a pairing of generators
$(\mathbf{x_i},\mathbf{y_i})$, so that for each pair there is a $d$-arrow
$\mathbf{x_i}\to\mathbf{y_i}$. In this case, there are no special cases in $H$,
so $H$ is exactly the map sending $\mathbf{y_i}$ to $\mathbf{x_i}$ and
$\mathbf{x_i}$ to 0 in each pair $(\mathbf{x_i},\mathbf{y_i})$.

\begin{table}[h]
  \centering
  \begin{tabular}{c|cc}\hline
    \backslashbox{$i$ strand}{$j$ strand} &
    left & right \\ \hline
    left & ($d$,$H$)-1 & $H$-3 \\
    right & $d$-4 & ($d$,$H$)-2 \\
  \end{tabular}
  \caption{Summary of cases for generators not of multiplicity one.}
  \label{table:summary2}
\end{table}


%% file: koszul34.tex
\clearpage
\section{Verification of the homotopy} \label{sec:proof}

\subsection{Multiplicity one}

In this section we verify that the map $H$ defined above satisfies the equation
$d\circ H+H\circ d=\mathbb{I}_M$ on the non-idempotent part of $M$, beginning
with the multiplicity one case. We do so by checking the equation
\begin{equation} \label{eq:homotopynonidem}
  d\circ H + H\circ d = \mathbb{I}_M
\end{equation}
on each of the eight types of generators. Depending on the type, the generator
appears on one side of some pair $(\mathbf{x_i},\mathbf{y_i})$. For every such
pair, there is a $d$-arrow $\mathbf{x_i}\to\mathbf{y_i}$ and an $H$-arrow
$\mathbf{y_i}\to\mathbf{x_i}$. So there is a natural identity term in $d\circ
H+H\circ d$, and it remains to check that all other terms sum to zero.

Write $H=H_{ord}+H_{sp}$, where $H_{sp}$ contains arrows coming from the three
special cases of $H$, and $H_{ord}$ is the remaining part of $H$. So for each
pair $(\mathbf{x_i},\mathbf{y_i})$ we have
$H_{ord}\mathbf{y_i}=\mathbf{x_i}$. It is also not difficult to determine
$\mathbf{y_i}$ from $\mathbf{x_i}$ - simply take the obvious differential around
$p+1$. We call a generator \emph{special} if it appears on either side of a pair
$(\mathbf{x_i},\mathbf{y_i})$ such that $H_{sp}\mathbf{y_i}\neq 0$. One easily
checks that a generator is special if and only if there is a strand ending at
$q$ on the left. Generators of type $H$-1, $H$-3, $d$-5 and $d$-8 can never be
special (the first two due to double horizontal on the left, the last two due to
a strand ending at $p+1$ on the left).

The strategy for verifying (\ref{eq:homotopynonidem}) is different for
generators on the $H$ side and the $d$ side, so we will describe them
separately. For a generator $\mathbf{y}$ on the $H$-side, we define three sets
of generators of $M$. Let $G_1$ be the set of terms in $d\mathbf{y}$ that are on
the $H$ side, $G_2$ be the set of all terms in $d(H_{ord}\mathbf{y})$ excluding
$\mathbf{y}$, and $G_3$ be the set of all terms in $d(H_{sp}\mathbf{y})$. For a
non-special $\mathbf{y}$, $G_3$ is of course empty. We will show that in this
case all generators in $G_1$ are also non-special, and that $H_{ord}$ induces a
bijection between $G_1$ and $G_2$. For special $\mathbf{y}$ we will show the
following:
\begin{itemize}
\item There is a generator $\mathbf{y_o}$ common to both $G_2$ and $G_3$.
\item All generators in $G_1$ are also special.
\item The map $H_{ord}$ induces a bijection between $G_1$ and $G_2\setminus
  \{\mathbf{y_o}\}$.
\item The map $H_{sp}$ induces a bijection between $G_1$ and $G_3\setminus
  \{\mathbf{y_o}\}$.
\end{itemize}
These imply that the terms in $d\circ H+H\circ d$ other than the natural
identity term sum to zero.

A large part of the bijection can be considered ``trivial'', since it involves
$d$-arrows that do not modify strands around $p+1$ or $q$. They can be carried
out in the ``same'' way on either side of an $H$-arrow. So a term in
$d\mathbf{y}$ coming from a $d$-arrow away from $p+1$ and $q$ corresponds to the
term obtained by performing the ``same'' move in $d(H_{ord}\mathbf{y})$ or
$d(H_{sp}\mathbf{y})$. Since $H$-arrows can introduce double horizontals at
$\{p+1,q\}$ on the right side, it can affect whether a strand ending at $q$ on
the left can be moved to the right in a $d$-arrow. This is essentially the
reason why special cases are needed.

The main checks are therefore for differentials involving strands around $p+1$
and $q$. The cases are described below, and the reader is advised to follow
along using the illustrations in Appendix \ref{sec:appendix}. In each
cancellation diagram in the appendix, the top-left generator is $\mathbf{y}$,
the top-right generator is either $H_{ord}\mathbf{y}$ or $H_{sp}\mathbf{y}$ (as
indicated on the top edge).  The bottom-left generator is an element in $G_1$,
and the bottom-right generator is the corresponding element in $G_2$ or
$G_3$. For generators that are special, there is one more diagram showing the
cancellation involving $\mathbf{y_o}$. In these diagrams, $\mathbf{y}$ is on the
top-left, $H_{ord}\mathbf{y}$ is on the bottom-left, $H_{sp}\mathbf{y}$ is on
the top-right, and $\mathbf{y}_{\mathbf{o}}$ is on the bottom-right.

We now turn to the case of a generator $\mathbf{x}$ on the $d$-side. The first
term in $d\circ H+H\circ d$ is zero. There is a distinguished term $\mathbf{y}$
in $d\mathbf{x}$ - the generator paired with $\mathbf{x}$. If $\mathbf{x}$ is
non-special, then $\mathbf{y}$ is also non-special, and we show that this is the
only term in $d\mathbf{x}$ on the $H$ side.  If $\mathbf{x}$ is special, then
$\mathbf{y}$ is also special, and we show there is exactly one other,
non-special, term $\mathbf{y'}$ in $d\mathbf{x}$ on the $H$ side, and that
$H_{sp}\mathbf{y}=H\mathbf{y'}$. This verifies the equation for $\mathbf{x}$.
In the appendix, we will show a diagram for each type of the last cancellation,
with $\mathbf{x}$ on the top-left, $\mathbf{y}$ on the bottom-left,
$\mathbf{y'}$ on the top-right, and $H_{sp}\mathbf{y}=H\mathbf{y'}$ on the
bottom-right.

We will assume throughout that in the key pair $\{p+1, q\}$ the point $q$ is
above $p+1$. The reader may check that in the other case the situation is the
same or simpler.

\subsubsection{Case 1}

\begin{figure}[h!tb]
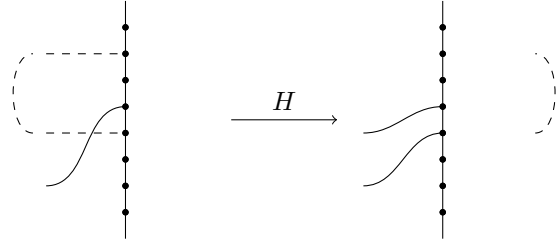

  \mapdiag{4}{7}{$H$}
  {\lstrand{2}{5}\ldoublehor{4}{7}}
  {\lstrand{2}{4}\lstrand{4}{5}}
  \caption{Homotopy, Case 1}
  \label{fig:homotopy1}
\end{figure}

In this case we consider generators $\mathbf{y}$ with key segment at left and
key pair at left as a double horizontal. The $H$-arrow $\mathbf{y}\to\mathbf{x}$
is shown in Figure \ref{fig:homotopy1}. There are no special generators. The
$d$-arrows starting at $\mathbf{y}$ involving the one strand around $p+1$
include:
\begin{itemize}
\item Merging with a strand above (Case 1.1)
\item Merging with a strand below (Case 1.2)
\item Moving a part not containing $p+1$ to the right (Case 1.3)
\item Moving a part containing $p+1$ to the right (Case 1.4).
\end{itemize}
The $d$-arrows starting at $\mathbf{x}$ involving the two strands shown include:
\begin{itemize}
\item Merging upper strand with a strand above (Case 1.1)
\item Merging lower strand with a strand below (Case 1.2)
\item Moving a part of the lower strand (Case 1.3)
\item Moving a part of the upper strand (Case 1.4).
\end{itemize}
Crossing the two strands gets back to $\mathbf{y}$.

\subsubsection{Case 2}

\begin{figure}[h!tb]
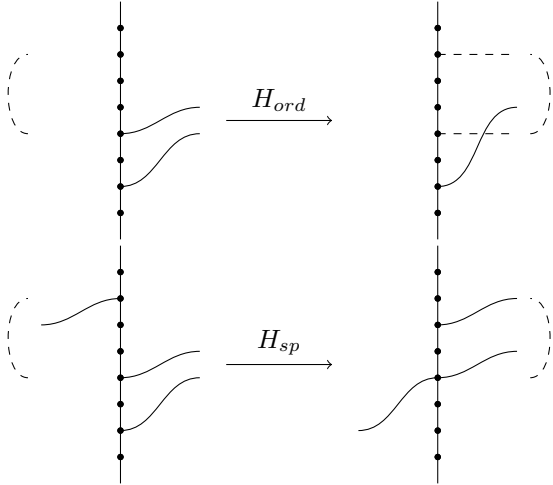

  \mapdiag{4}{7}{$H_{ord}$}
  {\rstrand{2}{4}\rstrand{4}{5}}
  {\rstrand{2}{5}\rdoublehor{4}{7}}
  \\ \\
  \mapdiag{4}{7}{$H_{sp}$}
  {\lstrand{6}{7}\rstrand{2}{4}\rstrand{4}{5}}
  {\lstrand{2}{4}\rstrand{4}{5}\rstrand{6}{7}}
  \caption{Homotopy, Case 2}
  \label{fig:homotopy2}
\end{figure}

In this case we consider generators $\mathbf{y}$ with key segment at right and
key pair at right as the start of a moving strand. The $H_{ord}$-arrow
$\mathbf{y}\to\mathbf{x}$ and the possible $H_{sp}$-arrow
$\mathbf{y}\to\mathbf{x}'$ are shown in Figure \ref{fig:homotopy2}. First
suppose the generator $\mathbf{y}$ is non-special (that is, no strand of
$\mathbf{y}$ end at $q$). The $d$-arrows starting at $\mathbf{y}$ involving the
two strands shown include:
\begin{itemize}
\item Splitting the lower strand (Case 2.1)
\item Splitting the upper strand (Case 2.2)
\item Moving a strand from left to attach to lower strand (Case 2.3).
\end{itemize}
The $d$-arrows starting at $\mathbf{x}$ involving the single strand shown
include:
\begin{itemize}
\item Splitting the single strand below $p+1$ (Case 2.1)
\item Splitting the single strand above $p+1$ (Case 2.2)
\item Moving a strand from left to attach to the single strand (Case 2.3).
\end{itemize}
Uncrossing the single strand gets back to $\mathbf{y}$.

Now suppose $\mathbf{y}$ is special. There is one more $d$-arrow starting at
$\mathbf{x}$. This is because the homotopy produces a double-horizontal at
$\{p+1, q\}$, so a strand ending at $q$ on the left can now move to the
right. This term cancels against a special $d$-arrow starting at $\mathbf{x}'$,
as shown in Case 2.4.

For the case when $\mathbf{y}$ is special, there is one more relevant strand -
the strand ending at $q$. The $d$-arrows starting at $\mathbf{y}$ involving this
strand include:
\begin{itemize}
\item Merging with a strand below (Case 2.5)
\item Moving part of the strand to the right (Case 2.6).
\end{itemize}
The $d$-arrows starting at $\mathbf{x}$ involving this strand are exactly the
same. The cancellations are shown in Cases 2.5a and 2.6a. The $d$-arrows
starting at $\mathbf{x}'$ involving this strand include:
\begin{itemize}
\item Attaching with a strand moved from the left (Case 2.5)
\item Splitting the strand (Case 2.6).
\end{itemize}
The cancellations are shown in Cases 2.5b and 2.6b. The $d$-arrows starting at
$\mathbf{x}'$ involving the two strands around $p+1$ include:
\begin{itemize}
\item Moving part of the lower strand to the right (Case 2.7b)
\item Splitting the upper strand (Case 2.8b)
\item Merging the lower strand with a strand below (Case 2.9b).
\end{itemize}
These cancel against $d$-arrows starting at $\mathbf{y}$ as in Cases 2.1, 2.2,
and 2.3, respectively.

\subsubsection{Case 3}

\begin{figure}[h!tb]
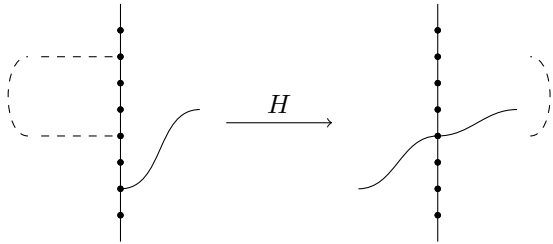

  \mapdiag{4}{7}{$H$}
  {\ldoublehor{4}{7}\rstrand{2}{5}}
  {\lstrand{2}{4}\rstrand{4}{5}}
  \caption{Homotopy, Case 3}
  \label{fig:homotopy3}
\end{figure}

In this case we consider generators $\mathbf{y}$ with key segment at right and
key pair at left as a double horizontal. The $H$-arrow $\mathbf{y}\to\mathbf{x}$
is shown in Figure \ref{fig:homotopy3}. Note if the single strand on the $H$
side ends exactly at $p+1$, then a double horizontal is produced on the right
side (however, since no strand can end at $q$ on the left, there are no special
cases). The $d$-arrows starting at $\mathbf{y}$ involving the single strand
include:
\begin{itemize}
\item Splitting at a point above $p+1$ (Case 3.1)
\item Splitting at a point below $p+1$ (Case 3.2)
\item Attaching with a strand moved from the left (Case 3.3).
\end{itemize}
The $d$-arrows starting at $\mathbf{x}$ include:
\begin{itemize}
\item Splitting the right strand (Case 3.1)
\item Moving part of the left strand to the right (Case 3.2)
\item Merging the left strand with a strand below (Case 3.3).
\end{itemize}
Moving all of the left strand to the right gets back to $\mathbf{y}$.

\subsubsection{Case 4}

\begin{figure}[h!tb]
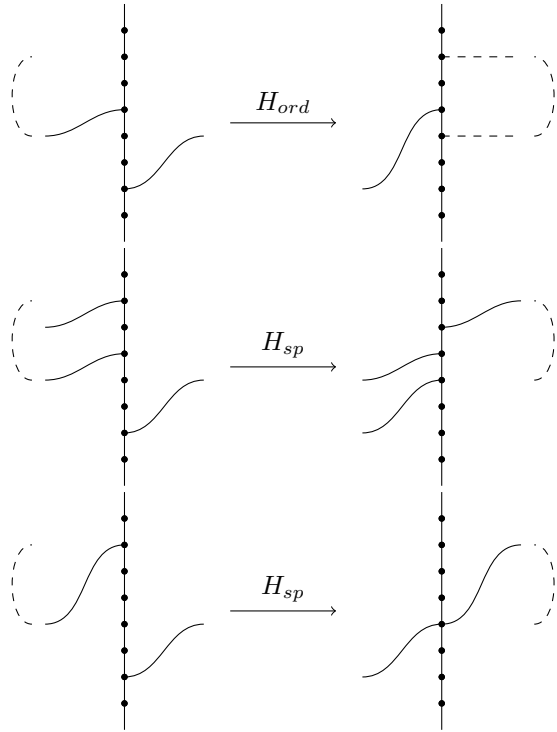

  \mapdiag{4}{7}{$H_{ord}$}
  {\lstrand{4}{5}\rstrand{2}{4}}
  {\lstrand{2}{5}\rdoublehor{4}{7}}
  \\ \\
  \mapdiag{4}{7}{$H_{sp}$}
  {\lstrand{4}{5}\lstrand{6}{7}\rstrand{2}{4}}
  {\lstrand{2}{4}\lstrand{4}{5}\rstrand{6}{7}}
  \\ \\
  \mapdiag{4}{7}{$H_{sp}$}
  {\lstrand{4}{7}\rstrand{2}{4}}
  {\lstrand{2}{4}\rstrand{4}{7}}
  \caption{Homotopy, Case 4}
  \label{fig:homotopy4}
\end{figure}

In this case we consider generators $\mathbf{y}$ with key segment at right and
key pair at left as the start of a moving strand. The $H_{ord}$-arrow
$\mathbf{y}\to\mathbf{x}$ and the two possible $H_{sp}$-arrows
$\mathbf{y}\to\mathbf{x}'$ are shown in Figure \ref{fig:homotopy4}. Note the
conditions for the two possible $H_{sp}$-arrows are mutually exclusive, so there
is at most one $H_{sp}$-arrow in any given case.

We first consider the case where $\mathbf{y}$ is non-special. The $d$-arrows
starting at $\mathbf{y}$ involving the two strands shown include:
\begin{itemize}
\item Splitting the right strand (Case 4.1)
\item Attaching the right strand with a strand moved from the left (Case 4.2)
\item Moving part of the left strand to the right (Case 4.3)
\item Merging the left strand with a strand above (Case 4.4).
\end{itemize}
The $d$-arrows starting at $\mathbf{x}$ involving the single strand shown
include:
\begin{itemize}
\item Moving to the right a part not containing $p+1$ (Case 4.1)
\item Moving to the right a part strictly containing $p+1$ (Case 4.3)
\item Merging with a strand below (Case 4.2)
\item Merging with a strand above (Case 4.4).
\end{itemize}
Moving a part containing $p+1$ on the boundary gets back to $\mathbf{y}$.

Now suppose $\mathbf{y}$ is special. Since $H$ creates a double horizontal at
$\{p+1,q\}$, there is one more $d$-arrow starting at $\mathbf{x}$, moving the
strand ending at $q$ to the right. This strand may or may not extend to $p+1$,
giving the two special cases. This cancels against a special $d$-arrow starting
at $\mathbf{x}'$ in both cases, as shown in Cases 4.5 and 4.6.

For the other cases, there is again one more relevant strand, ending at $q$.
The $d$-arrows starting at $\mathbf{y}$ involving this strand include:
\begin{itemize}
\item Merging with a strand below, resulting in a strand starting above $p+1$
  (Case 4.7)
\item Moving part of the strand to the right (Case 4.8)
\item Merging with a strand below, resulting in the strand $p+1\to q$ (Case 4.9)
\end{itemize}
The $d$-arrows starting at $\mathbf{x}$ are exactly the same for Cases 4.7 and
4.8. For Case 4.9 the strand ending at $q$ is merged with a longer strand (shown
in Case 4.9a). The $d$-arrows starting at $\mathbf{x}'$ involving this strand
include:
\begin{itemize}
\item Attaching a strand moved from the left, resulting in a strand starting
  above $p+1$ (Case 4.7b)
\item Splitting the strand (Case 4.8b)
\item Attaching a strand moved from the left, resulting in the strand $p+1\to q$
  (Case 4.9b).
\end{itemize}
Finally, we consider $d$-arrows starting at $\mathbf{x}'$ involving the two
strands around $p+1$. They include:
\begin{itemize}
\item Moving part of the lower strand to the right (Case 4.10b,c)
\item Joining the lower strand with a strand below (Case 4.11b,c)
\item Moving part of the upper strand to the right (first special case, Case
  4.12b)
\item Splitting the upper strand (second special case, Case 4.12c)
\item Joining the upper strand with a strand above (first special case, Case
  4.13b).
\end{itemize}
These cancel against $d$-arrows starting at $\mathbf{y}$ as in Cases 4.1 through
4.4.

\subsubsection{Case 5}

\begin{figure}[h!tb]
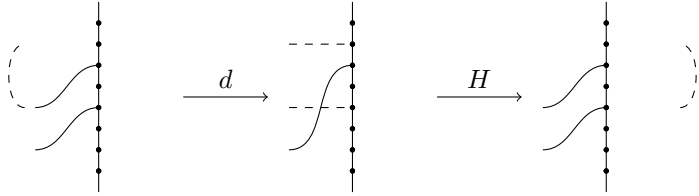

  \doublemapdiag{4}{7}{$d$}{$H$}
  {\lstrand{2}{4}\lstrand{4}{6}}
  {\ldoublehor{4}{7}\lstrand{2}{6}}
  {\lstrand{2}{4}\lstrand{4}{6}}
  \caption{Identity term in Case 5.}
  \label{fig:identity5}
\end{figure}

In this and the next three cases we consider generators on the $d$ side. In this
case, the generator $\mathbf{x}$ has key segment at left and key pair at left as
the start of a moving strand. There are no special generators in this case (due
to a strand ending at $p+1$ on the left). The only $d$-arrow
$\mathbf{x}\to\mathbf{y}$ to the $H$ side and the $H$-arrow
$\mathbf{y}\to\mathbf{x}$, giving the identity term, are shown in Figure
\ref{fig:identity5}.

The only other $d$-arrow starting at $\mathbf{x}$ that changes the type of the
generator is shown in Figure \ref{fig:excase5}, but we are still on the $d$
side.

\begin{figure}[h!tb]
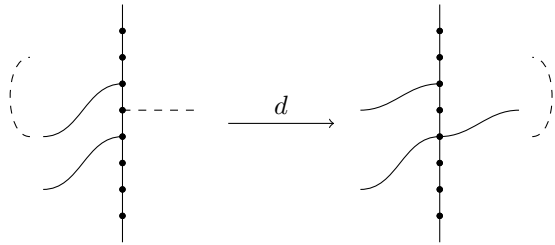

  \mapdiag{4}{7}{$d$}
  {\lstrand{2}{4}\lstrand{4}{6}\rsinglehor{5}}
  {\lstrand{2}{4}\lstrand{5}{6}\rstrand{4}{5}}
  \caption{A $d$-arrow that moves to a different type, but still on the $d$
    side.}
  \label{fig:excase5}
\end{figure}

\subsubsection{Case 6}

\begin{figure}[h!tb]
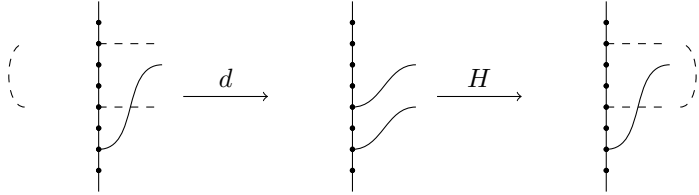

  \doublemapdiag{4}{7}{$d$}{$H$}
  {\rstrand{2}{6}\rdoublehor{4}{7}}
  {\rstrand{2}{4}\rstrand{4}{6}}
  {\rstrand{2}{6}\rdoublehor{4}{7}}
  \caption{Identity term in Case 6.}
  \label{fig:identity6}
\end{figure}

In this case, the generator $\mathbf{x}$ has key segment at right and key pair
at right as a double horizontal. If $\mathbf{x}$ is not special, Figure
\ref{fig:identity6} shows the only $d$-arrow $\mathbf{x}\to\mathbf{y}$ to the
$H$ side. If $\mathbf{x}$ is special --- that is, if there is a strand ending at
$q$ on the left, then the $d$-arrow moving that strand to the right will move
the double horizontal to the left, making a generator on the $H$ side. This
corresponds to the special case in Case 2 (Case 6.1).

\begin{figure}[h!tb]
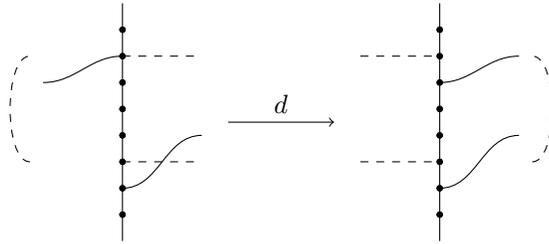

  \mapdiag{3}{7}{$d$}
  {\lstrand{6}{7}\rstrand{2}{4}\rdoublehor{3}{7}}
  {\ldoublehor{3}{7}\rstrand{6}{7}\rstrand{2}{4}}
  \caption{A $d$-arrow that maps to the $H$ side in Case 6, forcing a special
    case.}
  \label{fig:excase6}
\end{figure}

\subsubsection{Case 7}

\begin{figure}[h!tb]
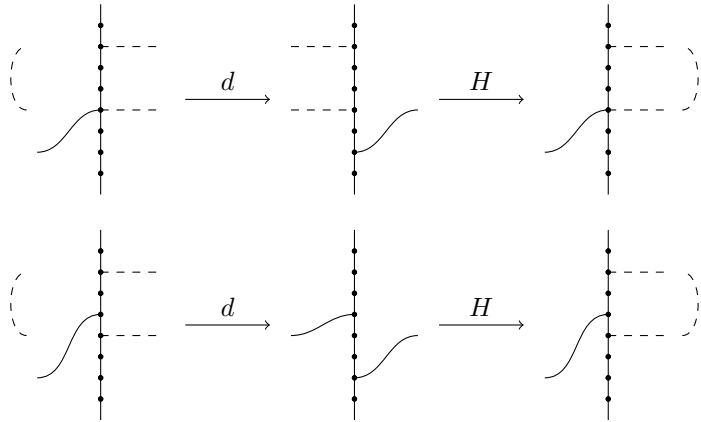

  \doublemapdiag{4}{7}{$d$}{$H$}
  {\lstrand{2}{4}\rdoublehor{4}{7}}
  {\ldoublehor{4}{7}\rstrand{2}{4}}
  {\lstrand{2}{4}\rdoublehor{4}{7}}
  \\ \\
  \doublemapdiag{4}{7}{$d$}{$H$}
  {\lstrand{2}{5}\rdoublehor{4}{7}}
  {\lstrand{4}{5}\rstrand{2}{4}}
  {\lstrand{2}{5}\rdoublehor{4}{7}}
  \caption{Identity term in Case 7.}
  \label{fig:identity7}
\end{figure}

In this case, the generator $\mathbf{x}$ has key segment at left and key pair at
right as a double horizontal. Figure \ref{fig:identity7} shows the obvious
$d$-arrow $\mathbf{x}\to\mathbf{y}$ to the $H$ side. If there is a strand ending
at $q$ (only possible in the second of the two cases in Figure
\ref{fig:identity7}), this strand can be moved to the right, moving the double
horizontal to the left and changing the type of the generator. There are two
ways this can happen, as shown in Figure \ref{fig:excase7}. They correspond to
the special cases in Case 4 (Cases 7.1 and 7.2).

\begin{figure} [h!tb]
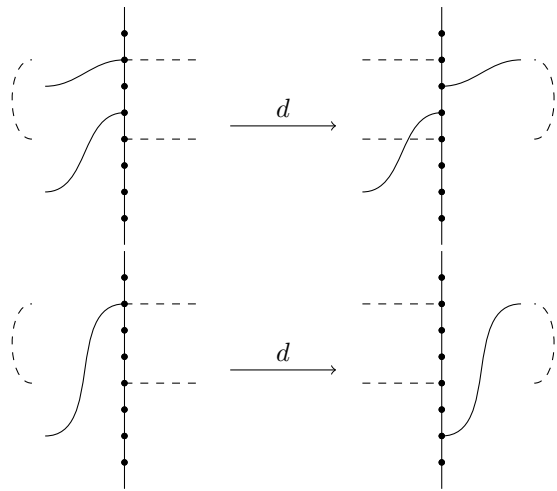

  \mapdiag{4}{7}{$d$}
  {\lstrand{2}{5}\lstrand{6}{7}\rdoublehor{4}{7}}
  {\lstrand{2}{5}\ldoublehor{4}{7}\rstrand{6}{7}}
  \\ \\
  \mapdiag{4}{7}{$d$}
  {\lstrand{2}{7}\rdoublehor{4}{7}}
  {\ldoublehor{4}{7}\rstrand{2}{7}}
  \caption{$d$-arrows that map to the $H$ side in Case 7, forcing two special
    cases.}
  \label{fig:excase7}
\end{figure}

\subsubsection{Case 8}

\begin{figure} [h!tb]
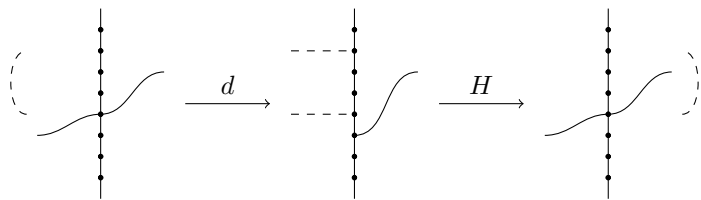

  \doublemapdiag{4}{7}{$d$}{$H$}
  {\lstrand{3}{4}\rstrand{4}{6}}
  {\ldoublehor{4}{7}\rstrand{3}{6}}
  {\lstrand{3}{4}\rstrand{4}{6}}
  \caption{Identity term in Case 8.}
  \label{fig:identity8}
\end{figure}

In this case, the generator $\mathbf{x}$ has key segment at left and key pair at
right as the start of a moving strand. The only $d$-arrow
$\mathbf{x}\to\mathbf{y}$ to the $H$ side is shown in Figure
\ref{fig:identity8}. There are no special generators.

This concludes the check for generators of multiplicity one.

\subsection{Other generators}

We now consider generators that do not have multiplicity one. They can be
classified into four types as in Table \ref{table:summary2}. We will again check
each type separately. The strategy is even simpler than before, as there are no
special cases. For a generator on the $H$ side, we check that there is a
bijection between the corresponding sets $G_1$ and $G_2$. For a generator on the
$d$ side, we check that there is exactly one differential to the $H$ side,
carrying it to the generator in the same pair.

The checks for Cases 1 and 2 in the definition of $H$ have a common ingredient
--- checking those terms where both the $d$-arrow and the $H$-arrow are
restricted to the same side.  For Case 2, the argument is the same as the one
used to show $H_*(\sla(\mathcal{Z}))=0$ on generators with multiplicity greater
than one. This is given in \cite[Section 4]{LOT10a}. While the proof there uses
a slightly different homotopy, the argument easily adapts to the homotopy we
give here. For Case 1 we use the dual of that argument. With these cases
covered, we will only consider terms in $d\circ H+H\circ d$ that involve both
sides.

\subsubsection{Case 1}

\begin{figure} [h!tb]
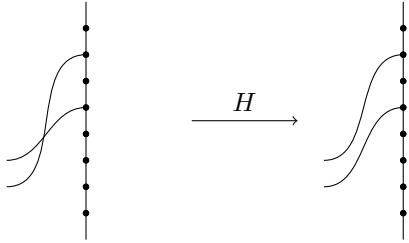

  \mapdiag{0}{0}{$H$}
  {\lstrand{2}{7}\lstrand{3}{5}}
  {\lstrand{2}{5}\lstrand{3}{7}}
  \caption{Homotopy in Case 1.}
  \label{fig:homotopym1}
\end{figure}

In this case both the $i$ strand and the $j$ strand are on the left. The
$H$-arrow $\mathbf{y}\to\mathbf{x}$ is shown in Figure \ref{fig:homotopym1}.

We begin by verifying the equation for $\mathbf{y}$. The $d$-arrows starting at
$\mathbf{y}$ include:
\begin{itemize}
\item Moving a part of the $j$ strand to the right. The end of that part can be:
  \subitem Lower than $i$ (Case 9.1)
  \subitem Higher than the endpoint of the $i$ strand (Case 9.2)
  \subitem Coinciding with the endpoint of the $i$ strand (Case 9.2')
  \subitem In between $i$ and the endpoint of the $i$ strand (Case 9.3)
\end{itemize}
Note that if part of the shorter strand is moved, we remain on the $d$ side
(Figure \ref{fig:excasem1}). As stated in the beginning of this section, we may
ignore $d$-arrows within the left side.

\begin{figure} [h!tb]
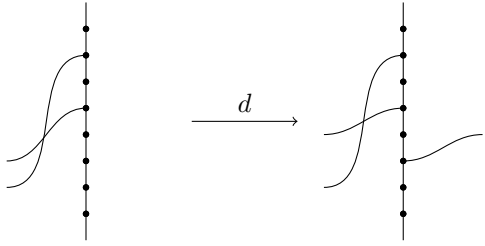

  \mapdiag{0}{0}{$d$}
  {\lstrand{2}{7}\lstrand{3}{5}}
  {\lstrand{2}{7}\lstrand{4}{5}\rstrand{3}{4}}
  \caption{A term in the differential that remains on the $d$ side.}
  \label{fig:excasem1}
\end{figure}

The $d$-arrows starting at $H\mathbf{y}=\mathbf{x}$ include:
\begin{itemize}
\item Moving part of the $j$ strand not containing $i$ (Case 9.1)
\item Moving part of the $i$ strand. The end of that part can be:
  \subitem Above the endpoint of the $j$ strand (Case 9.2)
  \subitem Coinciding with the endpoint of the $j$ strand (Case 9.2')
  \subitem Below the endpoint of the $j$ strand (Case 9.3).
\end{itemize}
Note it is impossible to move part of lower strand containing $j$, due to the
multiplication rule that double crossing gives zero.

Now we verify the equation for $\mathbf{x}$ on the $d$ side. Crossing the two
strands shown is the only way to go to a generator of type 1 on the $H$ side
(note we already covered the case where the $d$-arrow happens within the left
side). It is impossible to reach type 2 (both strands on the right). To reach
type 3, we need to arrange that the new $j$ strand is at right. This means we
will need to move a part of the current $j$ strand containing $i$ to the
right. However, this will produce a double crossing.

\subsubsection{Case 2}

\begin{figure} [h!tb]
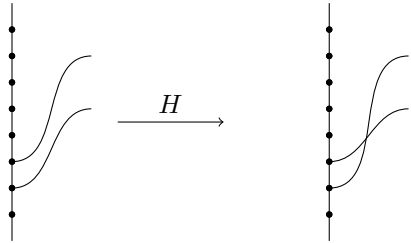

  \mapdiag{0}{0}{$H$}
  {\rstrand{2}{5}\rstrand{3}{7}}
  {\rstrand{2}{7}\rstrand{3}{5}}
  \caption{Homotopy in Case 2.}
  \label{fig:homotopym2}
\end{figure}

In this case, both the $i$ strand and the $j$ strand are on the right. The
$H$-arrow $\mathbf{y}\to\mathbf{x}$ is shown in Figure \ref{fig:homotopym2}.

For generators on the $H$ side, the $d$-arrows that involve only the right side
are again covered by the computation of $H_*(A(\mathcal{Z}))$. There is only one
other case that involves the displayed strands: moving a strand from the left to
attach to the $j$ strand (for both $d\mathbf{y}$ and $d(H\mathbf{y})$). This is
shown in Case 10.1. Note the interval $[j,i]$ cannot be occupied on the left
side.

For generators on the $d$ side, there is no way to reach the $H$ side other than
by $d$-arrows involving only the right side.

\subsubsection{Case 3}

\begin{figure} [h!tb]
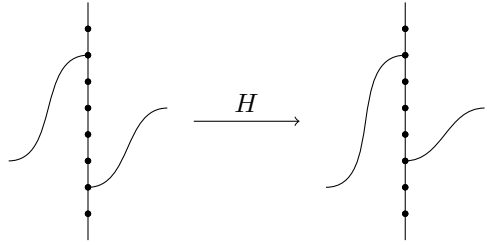

  \mapdiag{0}{0}{$H$}
  {\lstrand{3}{7}\rstrand{2}{5}}
  {\lstrand{2}{7}\rstrand{3}{5}}
  \caption{Homotopy in Case 3.}
  \label{fig:homotopym3}
\end{figure}

In this case, the $i$ strand is on the left and the $j$ strand is on the
right. The $H$-arrow $\mathbf{y}\to\mathbf{x}$ is shown in Figure
\ref{fig:homotopym3}. We consider the generator $\mathbf{y}$ in this case.

The $d$-arrows starting at $\mathbf{y}$ include:
\begin{itemize}
\item Merging the $i$ strand with a strand above (Case 11.1)
\item Splitting the $j$ strand below $i$ (Case 11.2)
\item Splitting the $j$ strand above $i$ (Case 11.3)
\item Attaching the $j$ strand with a strand moved from left (Case 11.4)
\item Moving part of the $i$ strand to the right, so that the new strand on the
  right does not cross the $j$ strand (Case 11.5).
\end{itemize}
If the new strand on the right does cross the $i$ strand, we remain on the $d$
side (Figure \ref{fig:excasem3}).

\begin{figure}
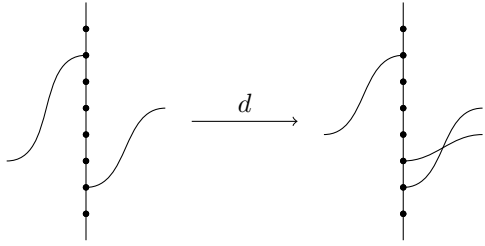

  \mapdiag{0}{0}{$d$}
  {\lstrand{3}{7}\rstrand{2}{5}}
  {\lstrand{4}{7}\rstrand{2}{5}\rstrand{3}{4}}
  \caption{A term in the differential that remains on the $d$ side.}
  \label{fig:excasem3}
\end{figure}

The $d$-arrows starting at $H\mathbf{y}=\mathbf{x}$ include:
\begin{itemize}
\item Merging the $j$ strand with a strand above (Case 11.1)
\item Merging the $j$ strand with a strand below (Case 11.4)
\item Splitting the $i$ strand (Case 11.3)
\item Moving part of $j$ strand not containing $i$ to the right (Case 11.2)
\item Moving part of $j$ strand to the right, so that the new strand on the
  right crosses the $i$ strand (Case 11.5).
\end{itemize}
Note if the new strand on the right interleaves, and does not cross the $i$
strand, the term is zero by double crossing (Figure \ref{fig:exnocasem3}).

\begin{figure}
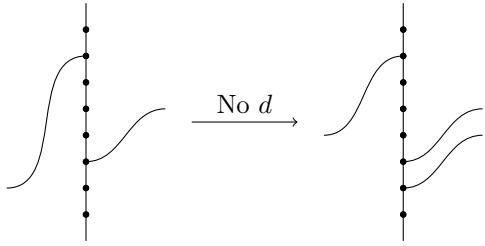

  \mapdiag{0}{0}{No $d$}
  {\lstrand{2}{7}\rstrand{3}{5}}
  {\lstrand{4}{7}\rstrand{3}{5}\rstrand{2}{4}}
  \caption{An example of a term that is not in the differential of $H\mathbf{x}$
    due to double crossing.}
  \label{fig:exnocasem3}
\end{figure}

\subsubsection{Case 4}

In this case we consider the generator $\mathbf{x}$ in the $H$-arrow in the
previous case, with the $i$ strand on the right and the $j$ strand on the
left. Beside the obvious $d$-arrow to $\mathbf{y}$, the only other possible way
to reach the $H$ side is by moving part of the $j$ strand containing $i$ to the
right. However, this is impossible due to double crossing (same example as in
Figure \ref{fig:exnocasem3}).

This concludes the check for generators with multiplicity greater than one, and
the verification that $H$ is indeed a homotopy.

\section{A rank-1 model of $\widehatit{CFAA}(\iz)$}\label{sec:typeaa}

Having found a homotopy $H$ on the chain complex underlying the
$A_\infty$-bimodule $\mathcal{M}$, we can use homological perturbation theory to
describe the smaller bimodule $\mathcal{N}$. We refer the reader to
\cite[Section 8.2]{LOT10c} on how to construct an $A_\infty$-bimodule action
using homological perturbation theory. The result is summarized in the following
theorem.

\begin{theorem} \label{thm:ndescription}
  The $A_\infty$-bimodule $\mathcal{N}_{\sla',\sla}$, described below, is
  homotopy equivalent to the bimodule $\mathcal{M}$ given in
  (\ref{eq:typeaaeq}), and is therefore a model of the invariant
  $\widehatit{CFAA}(\iz)_{\sla',\sla}$.

  The vector space $N$ underlying $\mathcal{N}$ is generated over $\mathbb{F}_2$
  by the indecomposable idempotents of $\sla$ (using the definitions in Section
  \ref{sec:chaincx}, $\mathcal{N}$ is a rank-1 bimodule). The arrows in the
  $A_\infty$-bimodule action correspond one-to-one with sequences
  $[a_{1,1},a_{1,2}],\dots,[a_{2n,1},a_{2n,2}]$ of generators of $\sla$ that
  satisfy the following three conditions:

  \begin{enumerate}
  \item $[a_{1,1},a_{1,2}]=[i',i]$ and $[a_{2n,1},a_{2n,2}]=[j',j]$ for some
    idempotents $i,j\in A$, and $i'=o(i)$, $j'=o(j)$.
  \item Each $[a_{2k,1},a_{2k,2}]$ is obtained from $[a_{2k-1,1},a_{2k-1,2}]$ by
    either factoring out some $\overline{b'}\in\sla$ from $a_{2k-1,1}$ on the
    right (so that $a_{2k-1,1}=a_{2k,1}\overline{b'}$ and
    $a_{2k-1,2}=a_{2k,2}$), or by multiplying $a_{2k-1,2}$ with some $b\in\sla$
    on the right (so that $a_{2k,2}=a_{2k-1,2}b$ and $a_{2k,1}=a_{2k-1,1}$). The
    elements $b'$ and $b$ are not necessarily equal between steps.
  \item Each $[a_{2k+1,1},a_{2k+1,2}]$ is obtained from $[a_{2k,1},a_{2k,2}]$ by
    one of the $H$ arrows described in Section \ref{sec:homotopy}.
  \end{enumerate}
  Let $b'_1,\dots,b'_p$ be the sequence of $b'\in\sla'$ and $b_1,\dots,b_q$ be
  the sequence of $b\in\sla$ used in condition 2. The arrow corresponding to
  this sequence is:
  \[ m_{1,p,q}([i]; b'_1,\dots,b'_p; b_1,\dots,b_q) \to [j]. \]

  Here $m_{1,p,q}$ is the part of the $A_\infty$-bimodule action on
  $\mathcal{N}$ with $p$ inputs on the $\sla'$ side and $q$ inputs on the $\sla$
  side. The idempotent of a generator $[i]$ of this $A_\infty$-bimodule is $i$
  on the $\sla$ side and $\overline{o(i)}$ on the $\sla'$ side.
\end{theorem}

For example, let $\slz$ be the genus 1 pointed matched circle. The following
sequence of $[a_{i,1},a_{i,2}]$

\begin{equation}
  \begin{smalldsd}\ldoublehor{1}{3}\rdoublehor{2}{4}\end{smalldsd}\to
  \begin{smalldsd}\ldoublehor{1}{3}\rstrand{2}{3}\end{smalldsd}\to
  \begin{smalldsd}\lstrand{2}{3}\rdoublehor{1}{3}\end{smalldsd}\to
  \begin{smalldsd}\lstrand{2}{3}\rstrand{1}{2}\end{smalldsd}\to
  \begin{smalldsd}\lstrand{1}{3}\rdoublehor{2}{4}\end{smalldsd}\to
  \begin{smalldsd}\ldoublehor{1}{3}\rdoublehor{2}{4}\end{smalldsd} \nonumber
\end{equation}

gives rise to the following arrow in the action:

\[
m_{1,1,2}\left(
  \left[\begin{smallsd}\doublehor{2}{4}\end{smallsd}\right];~
  \begin{smallsd}\stranddown{1}{3}\end{smallsd}~;~
  \begin{smallsd}\strandup{2}{3}\end{smallsd},
  \begin{smallsd}\strandup{1}{2}\end{smallsd}\right) \to
\left[\begin{smallsd}\doublehor{2}{4}\end{smallsd}\right],
\]
where $[i]$ for an idempotent $i\in\sla$ denotes the corresponding generator of
$\mathcal{N}$.

Since $\sla'$ and $\sla$ are opposite algebras, we can also consider
$\mathcal{N}$ as a left-right $A_\infty$-bimodule $_\sla\mathcal{N}_\sla$. We
will use this form of $\mathcal{N}$ in the next section.

The bimodule $\mathcal{N}$ inherits a relative grading from the larger model
$\mathcal{M}$. Since the generators in $\mathcal{N}$ are those in $\mathcal{M}$
composed of idempotents, we may choose the relative grading so that the grading
of all generators in $\mathcal{N}$ are zero (this is also the grading one would
obtain from the standard Heegaard diagram for the identity diffeomorphism,
starting with a zero grading for any one of the generators).

We now prove the main results of this chapter:

\begin{proof}[Proof of Theorem \ref{thm:nboxcfddprop}]
  Let
  \begin{equation} \label{eq:defofl}
    \tensor[^\sla]{\mathcal{L}}{_\sla} =
    \widehatit{CFAA}(\iz)_{\sla',\sla} \boxtimes
    \tensor[^{\sla,\sla'}]{\widehatit{CFDD}(\iz)}{}
    \simeq
    \mathcal{N}_{\sla',\sla}\boxtimes\tensor[^{\sla,\sla'}]{\widehatit{CFDD}(\iz)}{}.
  \end{equation}
  Since the vector spaces underlying both $\mathcal{N}$ and
  $\widehatit{CFDD}(\iz)$ are generated by idempotents, the vector space
  underlying $\mathcal{L}$ is also generated by idempotents. So $\mathcal{L}$ is
  a rank-1 bimodule. Furthermore, the type $\mathit{DA}$-bimodule action in
  $\mathcal{L}$ satisfies $\delta^1_1=0$ since all arrows in $\mathcal{N}$
  involve non-idempotent algebra inputs from both $\sla'$ and $\sla$. Hence
  Lemma 2.2.50 in \cite{LOT10a} applies, showing
  $\tensor[^\sla]{\mathcal{L}}{_\sla}=\tensor[^{\sla}]{[\phi]}{_{\sla}}$ for
  some $A_\infty$-algebra morphism $\phi:\sla\to\sla$.

  From the grading on $\mathcal{N}$ and $\widehatit{CFDD}(\iz)$, there is a
  relative grading on $\mathcal{L}$ with all generators having grading
  zero. This implies that the $A_\infty$ morphism $\phi$ respects the
  $G(\slz)$-grading on $\sla$. By classifying arrows in $\mathcal{N}$ involving
  only one length-1 chord on each side, it is clear that $\phi_1(\xi)=\xi$ for
  any length-1 chord $\xi$. By Proposition 4.11 in \cite{LOT10a}, we conclude
  that $\phi_1$ induces the identity map on homology. So $\phi$ is a
  quasi-isomorphism and $\mathcal{L}$ is quasi-invertible.
\end{proof}

\begin{corollary}
  $\widehatit{CFDD}(\iz)$ is quasi-invertible, hence the functor
  $\cdot\boxtimes\widehatit{CFDD}(\iz)$ induces an equivalence of categories
  between the category of right $A_\infty$-modules over $\sla(\slz)$, and the
  category of left type $D$ structures over $\sla(-\slz)$.
\end{corollary}

\begin{proof}
  With $\mathcal{L}$ as in (\ref{eq:defofl}), there is $\mathcal{L}'$ such that
  \[ ^\sla\mathbb{I}_\sla\simeq \mathcal{L}'\boxtimes\mathcal{L} \simeq
  \mathcal{L}'\boxtimes\mathcal{N}\boxtimes\widehatit{CFDD}(\iz), \] which means
  $\mathcal{L}'\boxtimes\mathcal{N}$ is the quasi-inverse to
  $\widehatit{CFDD}(\iz)$.
\end{proof}

\paragraph{Remark:} The stronger statement, that $\mathcal{N}$ is the
quasi-inverse of $\widehatit{CFDD}(\iz)$, will be proved combinatorially in
\cite{BZ2}.


%% file: koszul5.tex
\section{Examples of Koszul duality} \label{sec:exkoszul}

In this section, we use our description of $\mathcal{N}\simeq
\widehatit{CFAA}(\iz)$ to give an explicit $A_\infty$ morphism from
$\sla=\sla(\slz)$ to $\mathrm{Cob}(\sla)$, inducing an isomorphism on
homology. We will assume in this section that $\mathcal{N}$ is in fact the
quasi-inverse of $\widehatit{CFDD}(\iz)$. First, we review some material from
\cite{LOT10a} and \cite[Section 8]{LOT10b}.

\begin{definition}
  An \emph{augmentation} of a dg-algebra $A$ is a map $\epsilon: A\to\mathbf{k}$
  satisfying $\epsilon(1) = 1$ and $\epsilon(a_1a_2) =
  \epsilon(a_1)\epsilon(a_2)$. Given an augmentation, we let
  $A_+=\mathrm{ker}(\epsilon)$ be the \emph{augmentation ideal} of $A$.
\end{definition}

The strand algebra $\sla$ is augmented with the augmentation map $\epsilon$
sending each idempotent to itself and other generators to zero. So the
augmentation ideal $\sla_+$ is generated by the non-idempotents.

\begin{definition}
  Given an augmented dg-algebra $A$, the cobar resolution $\mathrm{Cob}(A)$ is
  defined as $T^*(A_+[1]^*)$, the tensor algebra of the dual of the augmentation
  ideal. This can be given the structure of a dg-algebra, with product being the
  one on the tensor algebra, and the differential consisting of the following
  arrows:
  \begin{itemize}
  \item $a_1^*\otimes\cdots\otimes b^*\otimes\cdots\otimes a_k^* \to
    a_1^*\otimes\cdots\otimes a_i^*\otimes\cdots\otimes a_k^*$, for each $i$ and
    term $b$ in $da_i$,
  \item $a_1^*\otimes\cdots\otimes a_i^*\otimes\cdots\otimes a_k^* \to
    a_1^*\otimes\cdots\otimes b^*\otimes b'^*\otimes\cdots\otimes a_k^*$, for
    each $i$ and generators $b, b'$ such that $bb'=a_i$.
  \end{itemize}
\end{definition}

For any augmented dg-algebra $A$, there is a type $\mathit{DD}$ bimodule
$^{\mathrm{Cob}(A)} K^A$, of rank 1 over $\mathbf{k}$, and with type
$\mathit{DD}$ action given by:
\[ \delta^1(\mathbf{1}) = \sum_{i} a_i^*\otimes\mathbf{1}\otimes a_i, \] where
the sum is over a set of generators of $A_+$ (the sides of the action are
reversed in comparison to \cite{LOT10b}, for ease of computation later). One may
check that this satisfies the structure equation for type $\mathit{DD}$
bimodules. The fact that this bimodule is always quasi-invertible shows that $A$
is Koszul dual to $\mathrm{Cob}(A)$ (see \cite[Proposition 8.12]{LOT10b}).

The following is contained in the proof of \cite[Proposition 8.11]{LOT10b}:
\begin{proposition}
  Let $_{\sla}\mathcal{N}_{\sla}$ be a rank-1 representative of
  $_{\sla}\widehatit{CFAA}(\iz)_{\sla}$. Then
  \[ \tensor*[^{\mathrm{Cob}(\sla)}] {K}{^{\sla}}
  \boxtimes\tensor*[_{\sla}]{\mathcal{N}}{_{\sla}} \] is of the form
  $^{\mathrm{Cob}(\sla)}[\phi]_{\sla}$, where $\phi$ is an $A_\infty$ morphism
  $\sla\to\mathrm{Cob}(\sla)$ that induces an isomorphism on homology.
\end{proposition}

In particular $\phi_1$ (the part of $\phi$ taking one input) maps
representatives of homology classes of $\sla$ to representatives of homology
classes of $\mathrm{Cob}(\sla)$. Unwinding all the definitions, we obtain the
following rule for computing this map:

\begin{proposition}
  Let $a\in\sla$ be a generator of $\sla$. Then terms in
  $\phi_1(a)\in\mathrm{Cob}(A)$ correspond to arrows in the $A_\infty$ action of
  $_\sla\mathcal{N}_\sla$ of the form
  \[ m_{p,1,1}(a_1, \dots, a_p, [i], a) \to [j], \] with the above arrow giving
  rise to the term
  \[ a_p^* \otimes \cdots \otimes a_1^* \] in $\phi_1(a)$.
\end{proposition}

We now give some examples. In each case, there is exactly one algebra input
$a\in\sla$ on the right. So under the notations of Theorem
\ref{thm:ndescription}, the sequence must start with $a_{2,2}=a$, then alternate
between applying an $H$-arrow and factoring from $a_{2k-1,1}$.

For a length-1 strand $a\in\sla$, we get $\phi_1(a)=a^*$. For the first
non-trivial case, we consider a length-2 strand, when none of the three points
are paired. There are two possible orderings of the two intervals in
$<_\slz$. If the upper interval comes first, the sequence

\begin{equation}
  \begin{dsdn}{3}\lsinglehor{2}\lsinglehor{3}\rsinglehor{1}\end{dsdn}\to
  \begin{dsdn}{3}\lsinglehor{2}\lsinglehor{3}\rstrand{1}{3}\end{dsdn}\to
  \begin{dsdn}{3}\lsinglehor{2}\lstrand{1}{3}\rsinglehor{3}\end{dsdn}\to
  \begin{dsdn}{3}\lsinglehor{1}\lsinglehor{2}\rsinglehor{3}\end{dsdn} \nonumber
\end{equation}
gives
\begin{equation} \label{eq:length2cob1}
  \phi_1\left(\begin{sdn}{3}\strandup{1}{3}\end{sdn}\right) =
  \begin{sdn}{3}\strandup{1}{3}\singlehor{2}\end{sdn},
\end{equation}
where we used the same diagram to represent $a$ and $a^*$.

If the lower interval comes first, the sequence

\begin{equation}
  \begin{dsdn}{3}\lsinglehor{2}\lsinglehor{3}\rsinglehor{1}\end{dsdn}\to
  \begin{dsdn}{3}\lsinglehor{2}\lsinglehor{3}\rstrand{1}{3}\end{dsdn}\to
  \begin{dsdn}{3}\lsinglehor{3}\lstrand{1}{2}\rstrand{2}{3}\end{dsdn}\to
  \begin{dsdn}{3}\lsinglehor{1}\lsinglehor{3}\rstrand{2}{3}\end{dsdn}\to
  \begin{dsdn}{3}\lsinglehor{1}\lstrand{2}{3}\rsinglehor{3}\end{dsdn}\to
  \begin{dsdn}{3}\lsinglehor{1}\lsinglehor{2}\rsinglehor{3}\end{dsdn} \nonumber
\end{equation}
gives
\begin{equation} \label{eq:length2cob2}
  \phi_1\left(\begin{sdn}{3}\strandup{1}{3}\end{sdn}\right) =
  \begin{sdn}{3}\strandup{2}{3}\singlehor{1}\end{sdn} \otimes
  \begin{sdn}{3}\strandup{1}{2}\singlehor{3}\end{sdn}.
\end{equation}

The right side of Equations (\ref{eq:length2cob1}) and (\ref{eq:length2cob2})
represent the same homology class in $\mathrm{Cob}(\sla)$, since their
difference is given by
\[ d\left(\begin{sdn}{3}\strandup{1}{2}\strandup{2}{3}\end{sdn}\right). \]

Now we consider a more complicated example, which shows the possible
complications that can arise in such a computation.

Consider a length 4 interval in a larger pointed matched circle, where the
second and fourth of the five points are paired. There are two possibilities for
the relative orderings of the four intervals in $<_\slz$. This is because the
second interval (counting from below) must immediately precede the third
interval, and the fourth interval must immediately precede the first
interval. The two possible relative orderings are given as follows. We will
simply refer to them as 1-4-3-2 and 3-2-1-4 from now on.
\[
\begin{sdn}[12pt]{5} \keypair{2}{4}{left}
  \draw (1.5, 4.5) node {1}; \draw (1.5, 3.5) node {4};
  \draw (1.5, 2.5) node {3}; \draw (1.5, 1.5) node {2};
\end{sdn} \quad \mathrm{and} \quad
\begin{sdn}[12pt]{5} \keypair{2}{4}{left}
  \draw (1.5, 4.5) node {3}; \draw (1.5, 3.5) node {2};
  \draw (1.5, 2.5) node {1}; \draw (1.5, 1.5) node {4};
\end{sdn}
\]

The following two generators $a_1, a_2$ represent the same homology class in
$\sla$:
\begin{equation} \label{eq:length5sameh} 
  a_1 = \begin{sdn}{5}\strandup{1}{2}\strandup{2}{5}\end{sdn}, \quad
  a_2 = \begin{sdn}{5}\strandup{1}{4}\strandup{4}{5}\end{sdn}.
\end{equation}

Consider first the ordering 1-4-3-2, and using the representative $a_1$. There
are three terms in $\phi_1(a_1)$. The first term comes from the sequence:

\begin{equation}
  \begin{dsdn}{5}\lsinglehor{3}\lsinglehor{5}\rsinglehor{1}
    \rdoublehor{2}{4}\end{dsdn}\to
  \begin{dsdn}{5}\lsinglehor{3}\lsinglehor{5}\rstrand{1}{2}
    \rstrand{2}{5}\end{dsdn}\to
  \begin{dsdn}{5}\lsinglehor{3}\lstrand{2}{5}\rstrand{1}{2}
    \rsinglehor{5}\end{dsdn}\to
  \begin{dsdn}{5}\ldoublehor{2}{4}\lsinglehor{3}\rstrand{1}{2}
    \rsinglehor{5}\end{dsdn}\to
  \begin{dsdn}{5}\lsinglehor{3}\lstrand{1}{2}\rdoublehor{2}{4}
    \rsinglehor{5}\end{dsdn}\to
  \begin{dsdn}{5}\lsinglehor{1}\lsinglehor{3}\rdoublehor{2}{4}
    \rsinglehor{5}\end{dsdn} \nonumber
\end{equation}

The second term comes from
\begin{equation}
  \begin{dsdn}{5}\lsinglehor{3}\lsinglehor{5}\rsinglehor{1}
    \rdoublehor{2}{4}\end{dsdn}\to
  \begin{dsdn}{5}\lsinglehor{3}\lsinglehor{5}\rstrand{1}{2}
    \rstrand{2}{5}\end{dsdn}\to
  \begin{dsdn}{5}\lsinglehor{3}\lstrand{2}{5}\rstrand{1}{2}
    \rsinglehor{5}\end{dsdn}\to
  \begin{dsdn}{5}\lsinglehor{3}\lstrand{2}{4}\rstrand{1}{2}
    \rsinglehor{5}\end{dsdn}\to
  \begin{dsdn}{5}\lsinglehor{3}\lstrand{1}{4}\rdoublehor{2}{4}
    \rsinglehor{5}\end{dsdn}\to
  \begin{dsdn}{5}\lsinglehor{1}\lsinglehor{3}\rdoublehor{2}{4}
    \rsinglehor{5}\end{dsdn} \nonumber
\end{equation}

The third term comes from
\begin{align}
  \begin{dsdn}{5}\lsinglehor{3}\lsinglehor{5}\rsinglehor{1}
    \rdoublehor{2}{4}\end{dsdn}\to
  \begin{dsdn}{5}\lsinglehor{3}\lsinglehor{5}\rstrand{1}{2}
    \rstrand{2}{5}\end{dsdn} &\to
  \begin{dsdn}{5}\lsinglehor{3}\lstrand{2}{5}\rstrand{1}{2}
    \rsinglehor{5}\end{dsdn}\to
  \begin{dsdn}{5}\lsinglehor{3}\lstrand{2}{4}\rstrand{1}{2}
    \rsinglehor{5}\end{dsdn}\to
  \begin{dsdn}{5}\lstrand{1}{2}\lsinglehor{3}\rstrand{2}{4}
    \rsinglehor{5}\end{dsdn}\to
  \begin{dsdn}{5}\lsinglehor{1}\lsinglehor{3}\rstrand{2}{4}
    \rsinglehor{5}\end{dsdn} \nonumber \\ &\to
  \begin{dsdn}{5}\lsinglehor{1}\lstrand{2}{3}\rstrand{3}{4}
    \rsinglehor{5}\end{dsdn} \to
  \begin{dsdn}{5}\lsinglehor{1}\ldoublehor{2}{4}\rstrand{3}{4}
    \rsinglehor{5}\end{dsdn} \to
  \begin{dsdn}{5}\lsinglehor{1}\lstrand{3}{4}\rdoublehor{2}{4}
    \rsinglehor{5}\end{dsdn} \to
  \begin{dsdn}{5}\lsinglehor{1}\lsinglehor{3}\rdoublehor{2}{4}
    \rsinglehor{5}\end{dsdn} \nonumber
\end{align}

Note the use of a special $H$-arrow in the fourth step of the last
sequence. This produces
\begin{equation} \label{eq:length5cob1}
  \phi_1\left(\begin{sdn}{5}\strandup{1}{2}\strandup{2}{5}\end{sdn}\right) =
  \begin{sdn}{5}\strandup{1}{2}\singlehor{3}\end{sdn} \otimes
  \begin{sdn}{5}\strandup{2}{5}\singlehor{3}\end{sdn} +
  \begin{sdn}{5}\strandup{1}{4}\singlehor{3}\end{sdn} \otimes
  \begin{sdn}{5}\strandup{4}{5}\singlehor{3}\end{sdn} +
  \begin{sdn}{5}\strandup{3}{4}\singlehor{1}\end{sdn} \otimes
  \begin{sdn}{5}\strandup{2}{3}\singlehor{1}\end{sdn} \otimes
  \begin{sdn}{5}\strandup{1}{2}\singlehor{3}\end{sdn} \otimes
  \begin{sdn}{5}\strandup{4}{5}\singlehor{3}\end{sdn}.
\end{equation}

If the ordering is 3-2-1-4, a straightforward sequence gives
\begin{equation} \label{eq:length5cob2}
  \phi_1\left(\begin{sdn}{5}\strandup{1}{2}\strandup{2}{5}\end{sdn}\right) =
  \begin{sdn}{5}\strandup{1}{2}\singlehor{3}\end{sdn} \otimes
  \begin{sdn}{5}\strandup{4}{5}\singlehor{3}\end{sdn} \otimes
  \begin{sdn}{5}\strandup{3}{4}\singlehor{5}\end{sdn} \otimes
  \begin{sdn}{5}\strandup{2}{3}\singlehor{5}\end{sdn}.
\end{equation}

The difference between the right sides of Equations (\ref{eq:length5cob1})
and (\ref{eq:length5cob2}) is:
\[ d\left(
  \begin{sdn}{5}\strandup{1}{2}\singlehor{3}\end{sdn} \otimes
  \begin{sdn}{5}\strandup{2}{3}\strandup{3}{5}\end{sdn} +
  \begin{sdn}{5}\strandup{1}{3}\strandup{3}{4}\end{sdn} \otimes
  \begin{sdn}{5}\strandup{4}{5}\singlehor{3}\end{sdn} +
  \begin{sdn}{5}\strandup{1}{2}\singlehor{3}\end{sdn} \otimes
  \begin{sdn}{5}\strandup{3}{4}\strandup{4}{5}\end{sdn} \otimes
  \begin{sdn}{5}\strandup{2}{3}\singlehor{5}\end{sdn} +
  \begin{sdn}{5}\strandup{3}{4}\singlehor{1}\end{sdn} \otimes
  \begin{sdn}{5}\strandup{1}{2}\strandup{2}{3}\end{sdn} \otimes
  \begin{sdn}{5}\strandup{4}{5}\singlehor{3}\end{sdn} \right). \]

Now we consider representative $a_2$ of the same homology class of $\sla$. If
the ordering is 1-4-3-2, the sequence is straightforward, giving
\begin{equation} \label{eq:length5cob3}
  \phi_1\left(\begin{sdn}{5}\strandup{1}{4}\strandup{4}{5}\end{sdn}\right) =
  \begin{sdn}{5}\strandup{3}{4}\singlehor{1}\end{sdn} \otimes
  \begin{sdn}{5}\strandup{2}{3}\singlehor{1}\end{sdn} \otimes
  \begin{sdn}{5}\strandup{1}{2}\singlehor{3}\end{sdn} \otimes
  \begin{sdn}{5}\strandup{4}{5}\singlehor{3}\end{sdn}.
\end{equation}

The difference between the right sides of Equations (\ref{eq:length5cob3}) and
(\ref{eq:length5cob1}) is:
\[ d\left(
  \begin{sdn}{5}\strandup{1}{5}\singlehor{3}\end{sdn}\right). \]

Finally, if the ordering is 3-2-1-4, one can check there is a single sequence,
which uses a special $H$-arrow in the fourth step:
\begin{align}
  \begin{dsdn}{5}\lsinglehor{3}\lsinglehor{5}\rsinglehor{1}
    \rdoublehor{2}{4}\end{dsdn}\to
  \begin{dsdn}{5}\lsinglehor{3}\lsinglehor{5}\rstrand{1}{4}
    \rstrand{4}{5}\end{dsdn} &\to
  \begin{dsdn}{5}\lsinglehor{5}\lstrand{1}{3}\rstrand{3}{4}
    \rstrand{4}{5}\end{dsdn}\to
  \begin{dsdn}{5}\lsinglehor{5}\lstrand{1}{2}\rstrand{3}{4}
    \rstrand{4}{5}\end{dsdn}\to
  \begin{dsdn}{5}\lstrand{3}{4}\lsinglehor{5}\rstrand{1}{2}
    \rstrand{4}{5}\end{dsdn}\to
  \begin{dsdn}{5}\lsinglehor{3}\lsinglehor{5}\rstrand{1}{2}
    \rstrand{4}{5}\end{dsdn} \nonumber \\ &\to
  \begin{dsdn}{5}\lsinglehor{3}\lstrand{4}{5}\rstrand{1}{2}
    \rsinglehor{5}\end{dsdn} \to
  \begin{dsdn}{5}\lsinglehor{3}\ldoublehor{2}{4}\rstrand{1}{2}
    \rsinglehor{5}\end{dsdn} \to
  \begin{dsdn}{5}\lsinglehor{3}\lstrand{1}{2}\rdoublehor{2}{4}
    \rsinglehor{5}\end{dsdn} \to
  \begin{dsdn}{5}\lsinglehor{1}\lsinglehor{3}\rdoublehor{2}{4}
    \rsinglehor{5}\end{dsdn} \nonumber
\end{align}
This produces
\begin{equation} \label{eq:length5cob4}
  \phi_1\left(\begin{sdn}{5}\strandup{1}{4}\strandup{4}{5}\end{sdn}\right) =
  \begin{sdn}{5}\strandup{1}{2}\singlehor{3}\end{sdn} \otimes
  \begin{sdn}{5}\strandup{4}{5}\singlehor{3}\end{sdn} \otimes
  \begin{sdn}{5}\strandup{3}{4}\singlehor{5}\end{sdn} \otimes
  \begin{sdn}{5}\strandup{2}{3}\singlehor{5}\end{sdn}.
\end{equation}
The right side is the same as that in Equation (\ref{eq:length5cob2}). This
finishes the consistency check that the same homology class in
$\mathrm{Cob}(\sla)$ is obtained using any valid local ordering of the
intervals, and any representative of the same homology class in $\sla$. In
general, one can expect the computation to be more complicated as the length of
elements in the homology class increases.


%% file: koszul_appendix.tex
\appendix
\clearpage
\section{Cancellation Diagrams}\label{sec:appendix}

In these diagrams, $d$-arrows are labeled $d$, ordinary $H$-arrows are labeled
$H$, and special $H$-arrows are labeled $H_{sp}$.

\begin{figure} [h]
\subfloat[Case 1.1] {
\canceldiag{4}{7}
{\lstrand{2}{5}\lstrand{5}{6}\ldoublehor{4}{7}}
{\lstrand{2}{4}\lstrand{4}{5}\lstrand{5}{6}}
{\lstrand{2}{6}\ldoublehor{4}{7}\lsinglehor{5}}
{\lstrand{2}{4}\lstrand{4}{6}\lsinglehor{5}}
}
\quad\quad
\subfloat[Case 1.2] {
\canceldiag{4}{7}
{\lstrand{2}{3}\lstrand{3}{6}\ldoublehor{4}{7}}
{\lstrand{2}{3}\lstrand{3}{4}\lstrand{4}{6}}
{\lstrand{2}{6}\ldoublehor{4}{7}\lsinglehor{3}}
{\lstrand{2}{4}\lstrand{4}{6}\lsinglehor{3}}
}
\caption{Cancellations in $d\circ H+H\circ d=\mathbb{I}_M$.}
\label{fig:cancellations1}
\end{figure}

\clearpage

\begin{figure}
\subfloat[Case 1.3] {
\canceldiag{4}{7}
{\lstrand{2}{6}\ldoublehor{4}{7}}
{\lstrand{2}{4}\lstrand{4}{6}}
{\lstrand{3}{6}\ldoublehor{4}{7}\rstrand{2}{3}}
{\lstrand{3}{4}\lstrand{4}{6}\rstrand{2}{3}}
}
\quad\quad
\subfloat[Case 1.4] {
\canceldiag{4}{7}
{\lstrand{2}{6}\ldoublehor{4}{7}}
{\lstrand{2}{4}\lstrand{4}{6}}
{\lstrand{5}{6}\ldoublehor{4}{7}\rstrand{2}{5}}
{\lstrand{5}{6}\lstrand{2}{4}\rstrand{4}{5}}
}
\caption{Cancellations in $d\circ H+H\circ d=\mathbb{I}_M$.}
\label{fig:cancellations2}
\end{figure}

\begin{figure}
\subfloat[Case 2.1] {
\canceldiag{4}{7}
{\rstrand{2}{4}\rstrand{4}{6}\rsinglehor{3}}
{\rstrand{2}{6}\rsinglehor{3}\rdoublehor{4}{7}}
{\rstrand{2}{3}\rstrand{3}{4}\rstrand{4}{6}}
{\rstrand{2}{3}\rstrand{3}{6}\rdoublehor{4}{7}}
}
\quad\quad
\subfloat[Case 2.2] {
\canceldiag{4}{7}
{\rstrand{2}{4}\rstrand{4}{6}\rsinglehor{5}}
{\rstrand{2}{6}\rsinglehor{5}\rdoublehor{4}{7}}
{\rstrand{2}{4}\rstrand{4}{5}\rstrand{5}{6}}
{\rstrand{2}{5}\rstrand{5}{6}\rdoublehor{4}{7}}
}
\caption{Cancellations in $d\circ H+H\circ d=\mathbb{I}_M$.}
\label{fig:cancellations3}
\end{figure}

\clearpage

\begin{figure}
\subfloat[Case 2.3] {
\canceldiag{5}{7}
{\lstrand{2}{3}\rstrand{3}{5}\rstrand{5}{6}}
{\lstrand{2}{3}\rstrand{3}{6}\rdoublehor{5}{7}}
{\lsinglehor{3}\rstrand{2}{5}\rstrand{5}{6}}
{\lsinglehor{3}\rdoublehor{5}{7}\rstrand{2}{6}}
}
\quad\quad
\setarrows{$H_{sp}$}{$H$}{$d$}{$d$}
\subfloat[Case 2.4] {
\canceldiag{4}{7}
{\lstrand{6}{7}\rstrand{2}{4}\rstrand{4}{5}}
{\lstrand{2}{4}\rstrand{4}{5}\rstrand{6}{7}}
{\lstrand{6}{7}\rstrand{2}{5}\rdoublehor{4}{7}}
{\rstrand{2}{5}\rstrand{6}{7}}
}
\resetarrows
\caption{Cancellations in $d\circ H+H\circ d=\mathbb{I}_M$.}
\label{fig:cancellations4}
\end{figure}

\begin{figure}
\subfloat[Case 2.5a] {
\canceldiag{3}{7}
{\lstrand{5}{6}\lstrand{6}{7}\rstrand{2}{3}\rstrand{3}{4}}
{\lstrand{5}{6}\lstrand{6}{7}\rstrand{2}{4}\rdoublehor{3}{7}}
{\lstrand{5}{7}\lsinglehor{6}\rstrand{2}{3}\rstrand{3}{4}}
{\lstrand{5}{7}\lsinglehor{6}\rstrand{2}{4}\rdoublehor{3}{7}}
}
\quad\quad
\sethspecial
\subfloat[Case 2.5b] {
\canceldiag{3}{7}
{\lstrand{5}{6}\lstrand{6}{7}\rstrand{2}{3}\rstrand{3}{4}}
{\lstrand{2}{3}\lstrand{5}{6}\rstrand{3}{4}\rstrand{6}{7}}
{\lstrand{5}{7}\lsinglehor{6}\rstrand{2}{3}\rstrand{3}{4}}
{\lstrand{2}{3}\lsinglehor{6}\rstrand{3}{4}\rstrand{5}{7}}
}
\resetarrows
\caption{Cancellations in $d\circ H+H\circ d=\mathbb{I}_M$.}
\label{fig:cancellations5}
\end{figure}

\clearpage

\begin{figure}
\subfloat[Case 2.6a] {
\canceldiag{3}{7}
{\lstrand{5}{7}\rstrand{2}{3}\rstrand{3}{4}\rsinglehor{6}}
{\lstrand{5}{7}\rstrand{2}{4}\rdoublehor{3}{7}\rsinglehor{6}}
{\lstrand{6}{7}\rstrand{5}{6}\rstrand{2}{3}\rstrand{3}{4}}
{\lstrand{6}{7}\rstrand{5}{6}\rstrand{2}{4}\rdoublehor{3}{7}}
}
\quad\quad
\sethspecial
\subfloat[Case 2.6b] {
\canceldiag{3}{7}
{\lstrand{5}{7}\rstrand{2}{3}\rstrand{3}{4}\rsinglehor{6}}
{\lstrand{2}{3}\rstrand{3}{4}\rstrand{5}{7}\rsinglehor{6}}
{\lstrand{6}{7}\rstrand{5}{6}\rstrand{2}{3}\rstrand{3}{4}}
{\lstrand{2}{3}\rstrand{3}{4}\rstrand{5}{6}\rstrand{6}{7}}
}
\resetarrows
\caption{Cancellations in $d\circ H+H\circ d=\mathbb{I}_M$.}
\label{fig:cancellations6}
\end{figure}

\begin{figure}
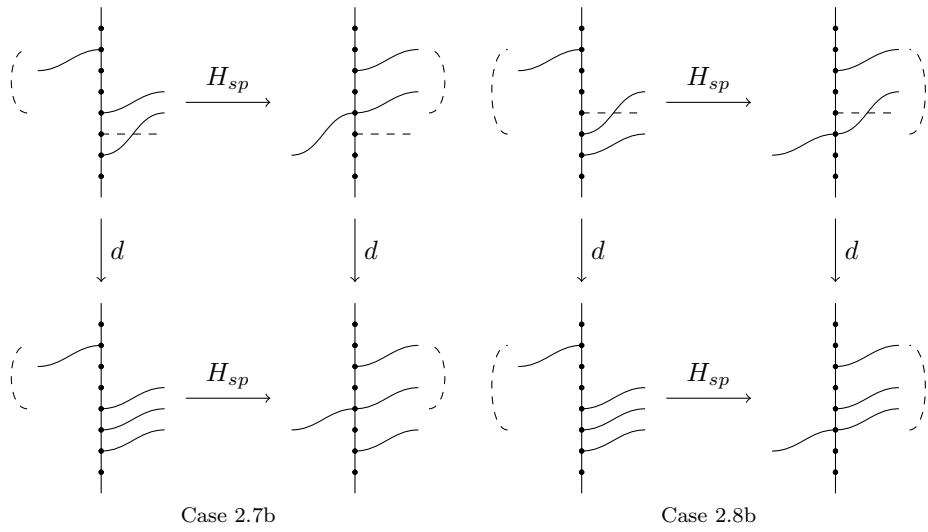

\sethspecial
\subfloat[Case 2.7b] {
\canceldiag{4}{7}
{\lstrand{6}{7}\rstrand{2}{4}\rstrand{4}{5}\rsinglehor{3}}
{\lstrand{2}{4}\rstrand{4}{5}\rstrand{6}{7}\rsinglehor{3}}
{\lstrand{6}{7}\rstrand{2}{3}\rstrand{3}{4}\rstrand{4}{5}}
{\lstrand{3}{4}\rstrand{2}{3}\rstrand{4}{5}\rstrand{6}{7}}
}
\subfloat[Case 2.8b] {
\canceldiag{3}{7}
{\lstrand{6}{7}\rstrand{2}{3}\rstrand{3}{5}\rsinglehor{4}}
{\lstrand{2}{3}\rstrand{3}{5}\rstrand{6}{7}\rsinglehor{4}}
{\lstrand{6}{7}\rstrand{2}{3}\rstrand{3}{4}\rstrand{4}{5}}
{\lstrand{2}{3}\rstrand{3}{4}\rstrand{4}{5}\rstrand{6}{7}}
}
\resetarrows
\caption{Cancellations in $d\circ H+H\circ d=\mathbb{I}_M$. Cases 2.7a and
2.8a are exactly the same as Cases 2.1 and 2.2.}
\label{fig:cancellations7}
\end{figure}

\clearpage

\begin{figure}
\sethspecial
\subfloat[Case 2.9b] {
\canceldiag{4}{7}
{\lstrand{2}{3}\lstrand{6}{7}\rstrand{3}{4}\rstrand{4}{5}}
{\lstrand{2}{3}\lstrand{3}{4}\rstrand{4}{5}\rstrand{6}{7}}
{\lstrand{6}{7}\lsinglehor{3}\rstrand{2}{4}\rstrand{4}{5}}
{\lstrand{2}{4}\lsinglehor{3}\rstrand{4}{5}\rstrand{6}{7}}
}
\resetarrows
\caption{Cancellations in $d\circ H+H\circ d=\mathbb{I}_M$. Case 2.9a are
exactly the same as Case 2.3.}
\label{fig:cancellations7a}
\end{figure}

\begin{figure}
\subfloat[Case 3.1] {
\canceldiag{4}{7}
{\ldoublehor{4}{7}\rstrand{2}{6}\rsinglehor{5}}
{\lstrand{2}{4}\rstrand{4}{6}\rsinglehor{5}}
{\ldoublehor{4}{7}\rstrand{2}{5}\rstrand{5}{6}}
{\lstrand{2}{4}\rstrand{4}{5}\rstrand{5}{6}}
}
\quad\quad
\subfloat[Case 3.2] {
\canceldiag{4}{7}
{\ldoublehor{4}{7}\rstrand{2}{6}\rsinglehor{3}}
{\lstrand{2}{4}\rstrand{4}{6}\rsinglehor{3}}
{\ldoublehor{4}{7}\rstrand{2}{3}\rstrand{3}{6}}
{\lstrand{3}{4}\rstrand{4}{6}\rstrand{2}{3}}
}
\caption{Cancellations in $d\circ H+H\circ d=\mathbb{I}_M$.}
\label{fig:cancellations8}
\end{figure}

\clearpage

\begin{figure}
\subfloat[Case 3.3] {
\canceldiag{4}{7}
{\ldoublehor{4}{7}\lstrand{2}{3}\rstrand{3}{6}}
{\lstrand{3}{4}\lstrand{2}{3}\rstrand{4}{6}}
{\ldoublehor{4}{7}\lsinglehor{3}\rstrand{2}{6}}
{\lstrand{2}{4}\lsinglehor{3}\rstrand{4}{6}}
}
\caption{Cancellations in $d\circ H+H\circ d=\mathbb{I}_M$.}
\label{fig:cancellations9}
\end{figure}

\begin{figure}
\subfloat[Case 4.1] {
\canceldiag{4}{7}
{\lstrand{4}{6}\rstrand{2}{4}\rsinglehor{3}}
{\lstrand{2}{6}\rdoublehor{4}{7}\rsinglehor{3}}
{\lstrand{4}{6}\rstrand{2}{3}\rstrand{3}{4}}
{\lstrand{3}{6}\rstrand{2}{3}\rdoublehor{4}{7}}
}
\quad\quad
\subfloat[Case 4.2] {
\canceldiag{4}{7}
{\lstrand{2}{3}\lstrand{4}{6}\rstrand{3}{4}}
{\lstrand{2}{3}\lstrand{3}{6}\rdoublehor{4}{7}}
{\lstrand{4}{6}\lsinglehor{3}\rstrand{2}{4}}
{\lstrand{2}{6}\lsinglehor{3}\rdoublehor{4}{7}}
}
\caption{Cancellations in $d\circ H+H\circ d=\mathbb{I}_M$.}
\label{fig:cancellations10}
\end{figure}

\clearpage

\begin{figure}
\subfloat[Case 4.3] {
\canceldiag{4}{7}
{\lstrand{4}{6}\rstrand{2}{4}\rsinglehor{5}}
{\lstrand{2}{6}\rdoublehor{4}{7}\rsinglehor{5}}
{\lstrand{5}{6}\rstrand{4}{5}\rstrand{2}{4}}
{\lstrand{5}{6}\rstrand{2}{5}\rdoublehor{4}{7}}
}
\quad\quad
\subfloat[Case 4.4] {
\canceldiag{4}{7}
{\lstrand{4}{5}\lstrand{5}{6}\rstrand{2}{4}}
{\lstrand{2}{5}\lstrand{5}{6}\rdoublehor{4}{7}}
{\lstrand{4}{6}\lsinglehor{5}\rstrand{2}{4}}
{\lstrand{2}{6}\lsinglehor{5}\rdoublehor{4}{7}}
}
\caption{Cancellations in $d\circ H+H\circ d=\mathbb{I}_M$.}
\label{fig:cancellations11}
\end{figure}

\begin{figure}
\setarrows{$H_{sp}$}{$H$}{$d$}{$d$}
\subfloat[Case 4.5] {
\canceldiag{4}{7}
{\lstrand{6}{7}\lstrand{4}{5}\rstrand{2}{4}}
{\lstrand{2}{4}\lstrand{4}{5}\rstrand{6}{7}}
{\lstrand{2}{5}\lstrand{6}{7}\rdoublehor{4}{7}}
{\lstrand{2}{5}\ldoublehor{4}{7}\rstrand{6}{7}}
}
\quad\quad
\subfloat[Case 4.6] {
\canceldiag{4}{7}
{\lstrand{4}{7}\rstrand{2}{4}}
{\lstrand{2}{4}\rstrand{4}{7}}
{\lstrand{2}{7}\rdoublehor{4}{7}}
{\ldoublehor{4}{7}\rstrand{2}{7}}
}
\resetarrows
\caption{Cancellations in $d\circ H+H\circ d=\mathbb{I}_M$.}
\label{fig:cancellations12}
\end{figure}

\clearpage

\begin{figure}
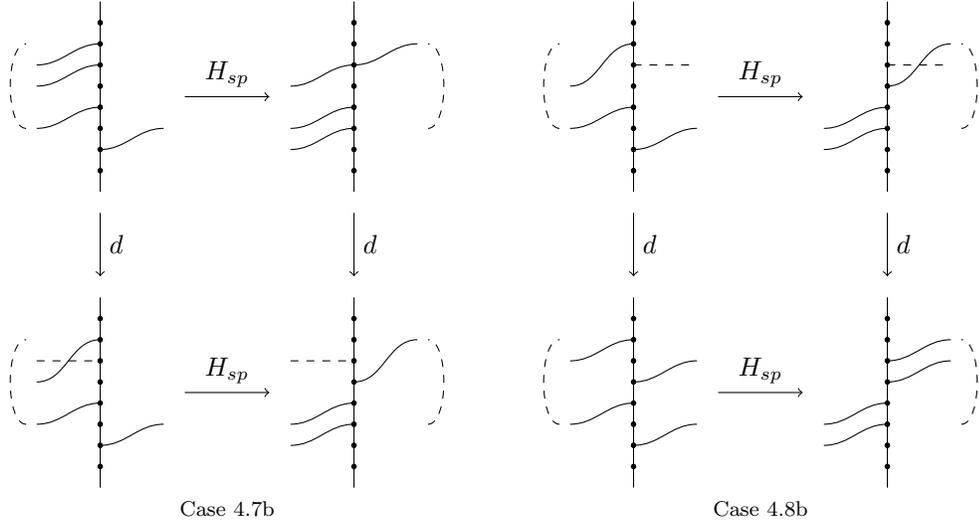

\sethspecial
\subfloat[Case 4.7b] {
\canceldiag{3}{7}
{\lstrand{3}{4}\lstrand{5}{6}\lstrand{6}{7}\rstrand{2}{3}}
{\lstrand{2}{3}\lstrand{3}{4}\lstrand{5}{6}\rstrand{6}{7}}
{\lstrand{3}{4}\lstrand{5}{7}\lsinglehor{6}\rstrand{2}{3}}
{\lstrand{2}{3}\lstrand{3}{4}\lsinglehor{6}\rstrand{5}{7}}
}
\quad\quad
\subfloat[Case 4.8b] {
\canceldiag{3}{7}
{\lstrand{3}{4}\lstrand{5}{7}\rstrand{2}{3}\rsinglehor{6}}
{\lstrand{2}{3}\lstrand{3}{4}\rstrand{5}{7}\rsinglehor{6}}
{\lstrand{3}{4}\lstrand{6}{7}\rstrand{5}{6}\rstrand{2}{3}}
{\lstrand{2}{3}\lstrand{3}{4}\rstrand{5}{6}\rstrand{6}{7}}
}
\resetarrows
\caption{Cancellations in $d\circ H+H\circ d=\mathbb{I}_M$. Cases 4.7a and 4.8a are similar to Cases 2.5a and 2.6a.}
\label{fig:cancellations16}
\end{figure}

\begin{figure}
\subfloat[Case 4.9a] {
\canceldiag{4}{7}
{\lstrand{4}{5}\lstrand{5}{7}\rstrand{2}{4}}
{\lstrand{2}{5}\lstrand{5}{7}\rdoublehor{4}{7}}
{\lstrand{4}{7}\lsinglehor{5}\rstrand{2}{4}}
{\lstrand{2}{7}\lsinglehor{5}\rdoublehor{4}{7}}
}
\quad\quad
\sethspecial
\subfloat[Case 4.9b] {
\canceldiag{4}{7}
{\lstrand{4}{5}\lstrand{5}{7}\rstrand{2}{4}}
{\lstrand{2}{4}\lstrand{4}{5}\rstrand{5}{7}}
{\lstrand{4}{7}\lsinglehor{5}\rstrand{2}{4}}
{\lstrand{2}{4}\lsinglehor{5}\rstrand{4}{7}}
}
\resetarrows
\caption{Cancellations in $d\circ H+H\circ d=\mathbb{I}_M$.}
\label{fig:cancellations17}
\end{figure}

\clearpage

\begin{figure}
\sethspecial
\subfloat[Case 4.10b] {
\canceldiag{4}{7}
{\lstrand{4}{5}\lstrand{6}{7}\rstrand{2}{4}\rsinglehor{3}}
{\lstrand{4}{5}\lstrand{2}{4}\rstrand{6}{7}\rsinglehor{3}}
{\lstrand{4}{5}\lstrand{6}{7}\rstrand{2}{3}\rstrand{3}{4}}
{\lstrand{3}{4}\lstrand{4}{5}\rstrand{6}{7}\rstrand{2}{3}}
}
\quad\quad
\subfloat[Case 4.10c] {
\canceldiag{4}{7}
{\lstrand{4}{7}\rstrand{2}{4}\rsinglehor{3}}
{\lstrand{2}{4}\rstrand{4}{7}\rsinglehor{3}}
{\lstrand{4}{7}\rstrand{2}{3}\rstrand{3}{4}}
{\lstrand{3}{4}\rstrand{2}{3}\rstrand{4}{7}}
}
\resetarrows
\caption{Cancellations in $d\circ H+H\circ d=\mathbb{I}_M$. Case 4.10a is same as Case 4.1.}
\label{fig:cancellations13}
\end{figure}

\begin{figure}
\sethspecial
\subfloat[Case 4.11b] {
\canceldiag{4}{7}
{\lstrand{6}{7}\lstrand{4}{5}\lstrand{2}{3}\rstrand{3}{4}}
{\lstrand{4}{5}\lstrand{3}{4}\lstrand{2}{3}\rstrand{6}{7}}
{\lstrand{4}{5}\lstrand{6}{7}\lsinglehor{3}\rstrand{2}{4}}
{\lstrand{2}{4}\lsinglehor{3}\lstrand{4}{5}\rstrand{6}{7}}
}
\quad\quad
\subfloat[Case 4.11c] {
\canceldiag{4}{7}
{\lstrand{4}{7}\lstrand{2}{3}\rstrand{3}{4}}
{\lstrand{2}{3}\lstrand{3}{4}\rstrand{4}{7}}
{\lstrand{4}{7}\lsinglehor{3}\rstrand{2}{4}}
{\lstrand{2}{4}\lsinglehor{3}\rstrand{4}{7}}
}
\resetarrows
\caption{Cancellations in $d\circ H+H\circ d=\mathbb{I}_M$. Case 4.11a is same as Case 4.2.}
\label{fig:cancellations14}
\end{figure}

\clearpage

\begin{figure}
\sethspecial
\subfloat[Case 4.12b] {
\canceldiag{3}{7}
{\lstrand{3}{5}\lstrand{6}{7}\rstrand{2}{3}\rsinglehor{4}}
{\lstrand{3}{5}\lstrand{2}{3}\rstrand{6}{7}\rsinglehor{4}}
{\lstrand{4}{5}\lstrand{6}{7}\rstrand{2}{3}\rstrand{3}{4}}
{\lstrand{2}{3}\lstrand{4}{5}\rstrand{3}{4}\rstrand{6}{7}}
}
\quad\quad
\subfloat[Case 4.12c] {
\canceldiag{3}{7}
{\lstrand{3}{7}\rstrand{2}{3}\rsinglehor{4}}
{\lstrand{2}{3}\rstrand{3}{7}\rsinglehor{4}}
{\lstrand{4}{7}\rstrand{2}{3}\rstrand{3}{4}}
{\lstrand{2}{3}\rstrand{3}{4}\rstrand{4}{7}}
}
\resetarrows
\caption{Cancellations in $d\circ H+H\circ d=\mathbb{I}_M$. Case 4.12a is same as Case 4.3.}
\label{fig:cancellations15}
\end{figure}

\begin{figure}
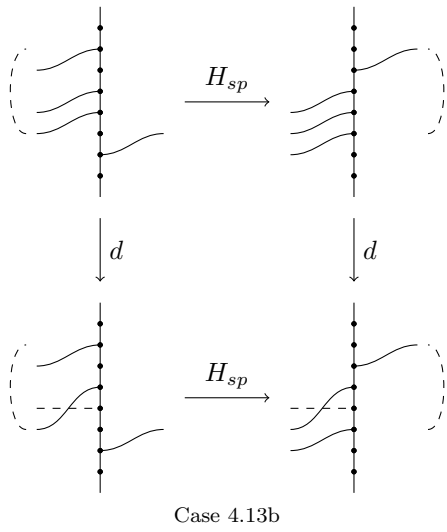

\sethspecial
\subfloat[Case 4.13b] {
\canceldiag{3}{7}
{\lstrand{3}{4}\lstrand{4}{5}\lstrand{6}{7}\rstrand{2}{3}}
{\lstrand{2}{3}\lstrand{3}{4}\lstrand{4}{5}\rstrand{6}{7}}
{\lstrand{3}{5}\lsinglehor{4}\lstrand{6}{7}\rstrand{2}{3}}
{\lstrand{2}{3}\lstrand{3}{5}\lsinglehor{4}\rstrand{6}{7}}
}
\resetarrows
\caption{Cancellations in $d\circ H+H\circ d=\mathbb{I}_M$. Case 4.13a is same as Case 4.4.
Case 4.13c, with the second special case, cannot occur.}
\label{fig:cancellations15a}
\end{figure}

\clearpage

\begin{figure}
\setarrows{$d$}{$d$}{$H$}{$H_{sp}$}
\subfloat[Case 6.1] {
\canceldiag{4}{7}
{\lstrand{6}{7}\rstrand{2}{5}\rdoublehor{4}{7}}
{\ldoublehor{4}{7}\rstrand{2}{5}\rstrand{6}{7}}
{\lstrand{6}{7}\rstrand{2}{4}\rstrand{4}{5}}
{\lstrand{2}{4}\rstrand{4}{5}\rstrand{6}{7}}
}
\resetarrows
\caption{Cancellations in $d\circ H+H\circ d=\mathbb{I}_M$.}
\label{fig:cancellations18}
\end{figure}

\begin{figure}
\setarrows{$d$}{$d$}{$H$}{$H_{sp}$}
\subfloat[Case 7.1] {
\canceldiag{4}{7}
{\lstrand{6}{7}\lstrand{2}{5}\rdoublehor{4}{7}}
{\lstrand{2}{5}\ldoublehor{4}{7}\rstrand{6}{7}}
{\lstrand{4}{5}\lstrand{6}{7}\rstrand{2}{4}}
{\lstrand{2}{4}\lstrand{4}{5}\rstrand{6}{7}}
}
\quad\quad
\subfloat[Case 7.2] {
\canceldiag{4}{7}
{\lstrand{2}{7}\rdoublehor{4}{7}}
{\ldoublehor{4}{7}\rstrand{2}{7}}
{\lstrand{4}{7}\rstrand{2}{4}}
{\lstrand{2}{4}\rstrand{4}{7}}
}
\resetarrows
\caption{Cancellations in $d\circ H+H\circ d=\mathbb{I}_M$.}
\label{fig:cancellations19}
\end{figure}

\clearpage

\begin{figure}
\subfloat[Case 9.1] {
\canceldiag{0}{0}
{\lstrand{2}{6}\lstrand{4}{5}\rsinglehor{3}}
{\lstrand{2}{5}\lstrand{4}{6}\rsinglehor{3}}
{\lstrand{3}{6}\lstrand{4}{5}\rstrand{2}{3}}
{\lstrand{3}{5}\lstrand{4}{6}\rstrand{2}{3}}
}
\quad\quad
\subfloat[Case 9.2] {
\canceldiag{0}{0}
{\lstrand{2}{6}\lstrand{3}{4}}
{\lstrand{2}{4}\lstrand{3}{6}}
{\lstrand{3}{4}\lstrand{5}{6}\rstrand{2}{5}}
{\lstrand{2}{4}\lstrand{5}{6}\rstrand{3}{5}}
}
\caption{Cancellations in $d\circ H+H\circ d=\mathbb{I}_M$.}
\label{fig:cancellations20}
\end{figure}

\begin{figure}
\subfloat[Case 9.2'] {
\canceldiag{0}{0}
{\lstrand{2}{6}\lstrand{3}{4}}
{\lstrand{3}{6}\lstrand{2}{4}}
{\lstrand{3}{4}\lstrand{4}{6}\rstrand{2}{4}}
{\lstrand{2}{4}\lstrand{4}{6}\rstrand{3}{4}}
}
\quad\quad
\subfloat[Case 9.3] {
\canceldiag{0}{0}
{\lstrand{2}{6}\lstrand{3}{5}}
{\lstrand{2}{5}\lstrand{3}{6}}
{\lstrand{3}{5}\lstrand{4}{6}\rstrand{2}{4}}
{\lstrand{2}{5}\lstrand{4}{6}\rstrand{3}{4}}
}
\caption{Cancellations in $d\circ H+H\circ d=\mathbb{I}_M$.}
\label{fig:cancellations21}
\end{figure}

\clearpage

\begin{figure}
\subfloat[Case 10.1] {
\canceldiag{0}{0}
{\lstrand{2}{3}\rstrand{3}{5}\rstrand{4}{6}}
{\lstrand{2}{3}\rstrand{3}{6}\rstrand{4}{5}}
{\lsinglehor{3}\rstrand{2}{5}\rstrand{4}{6}}
{\lsinglehor{3}\rstrand{2}{6}\rstrand{4}{5}}
}
\caption{Cancellations in $d\circ H+H\circ d=\mathbb{I}_M$.}
\label{fig:cancellations22}
\end{figure}

\begin{figure}
\subfloat[Case 11.1] {
\canceldiag{0}{0}
{\lstrand{3}{5}\lstrand{5}{6}\rstrand{2}{4}}
{\lstrand{2}{5}\lstrand{5}{6}\rstrand{3}{4}}
{\lstrand{3}{6}\lsinglehor{5}\rstrand{2}{4}}
{\lstrand{2}{6}\lsinglehor{5}\rstrand{3}{4}}
}
\quad\quad
\subfloat[Case 11.2] {
\canceldiag{0}{0}
{\lstrand{4}{6}\rstrand{2}{5}\rsinglehor{3}}
{\lstrand{2}{6}\rstrand{4}{5}\rsinglehor{3}}
{\lstrand{4}{6}\rstrand{2}{3}\rstrand{3}{5}}
{\lstrand{3}{6}\rstrand{2}{3}\rstrand{4}{5}}
}
\caption{Cancellations in $d\circ H+H\circ d=\mathbb{I}_M$.}
\label{fig:cancellations23}
\end{figure}

\clearpage

\begin{figure}
\subfloat[Case 11.3] {
\canceldiag{0}{0}
{\lstrand{3}{6}\rstrand{2}{5}\rsinglehor{4}}
{\lstrand{2}{6}\rstrand{3}{5}\rsinglehor{4}}
{\lstrand{3}{6}\rstrand{2}{4}\rstrand{4}{5}}
{\lstrand{2}{6}\rstrand{3}{4}\rstrand{4}{5}}
}
\quad\quad
\subfloat[Case 11.4] {
\canceldiag{0}{0}
{\lstrand{2}{3}\lstrand{4}{6}\rstrand{3}{5}}
{\lstrand{2}{3}\lstrand{3}{6}\rstrand{4}{5}}
{\lstrand{4}{6}\lsinglehor{3}\rstrand{2}{5}}
{\lstrand{2}{6}\lsinglehor{3}\rstrand{4}{5}}
}
\caption{Cancellations in $d\circ H+H\circ d=\mathbb{I}_M$.}
\label{fig:cancellations24}
\end{figure}

\begin{figure}
\subfloat[Case 11.5] {
\canceldiag{0}{0}
{\lstrand{3}{6}\rstrand{2}{4}\rsinglehor{5}}
{\lstrand{2}{6}\rstrand{3}{4}\rsinglehor{5}}
{\lstrand{5}{6}\rstrand{2}{4}\rstrand{3}{5}}
{\lstrand{5}{6}\rstrand{2}{5}\rstrand{3}{4}}
}
\caption{Cancellations in $d\circ H+H\circ d=\mathbb{I}_M$.}
\label{fig:cancellations25}
\end{figure}
